\documentclass[11pt]{amsart}
\usepackage{geometry,amsthm,amsmath,amssymb, graphicx, float, enumerate,mathtools,url,color,array,colortbl,xypic}
\usepackage{subfigure,caption,enumitem}
\usepackage[numbers]{natbib}
\restylefloat{table}
\restylefloat{figure}

\title{Binary Parseval frames from group orbits}
\author[R. P. Mendez, B. G. Bodmann,  Z. J. Baker, M. G. Bullock, and J. E. McLaney]{Robert P. Mendez, Bernhard G. Bodmann, Zachery J. Baker, Micah G. Bullock,\\ and Jacob E. McLaney}
\thanks{The research in this paper was partially supported by NSF grant DMS-1412524}

\swapnumbers
\newtheorem{thm}{Theorem}[section]

\newtheorem{cor}[thm]{Corollary}

\newtheorem{lemma}[thm]{Lemma}
\newtheorem{prop}[thm]{Proposition}

\theoremstyle{definition}
\newtheorem{defn}[thm]{Definition}
\theoremstyle{remark}
\newtheorem{rem}[thm]{Remark}

\newtheorem{ex}[thm]{Example}
\newtheorem{exs}[thm]{Examples}

\newcommand{\bb}[1]{\mathbb{#1}}
\newcommand{\cl}[1]{\mathcal{#1}}

\newcommand{\normz}[1]{\left\|#1\right\|_0}

\newcommand{\abs}[1]{\left|#1\right|}
\newcommand{\dotpRL}[2]{\left\langle #1,#2\right\rangle}

\newcommand{\nCksm}[2]{\left(\!    \begin{smallmatrix} #1 \\ #2    \end{smallmatrix}  \!\right)}
\newcommand{\nCkGen}[1]{\left(\!    \begin{smallmatrix} #1 \end{smallmatrix}  \!\right)}
\newcommand{\cVn}[1]{\begin{smallmatrix} #1 \end{smallmatrix} } 
 
\newcommand{\colVecsm}[1]{\left(\!    \begin{smallmatrix} \scriptsize #1  \end{smallmatrix}  \!\right)}

\newcommand{\numto}[1]{\{1,2, \dots, #1 \}}

\def\vphantFour{\vphantom{\cVn{0\\0\\0\\0 }}}

\def\N{\mathbb{N}}
\def\Z{\mathbb{Z}}

\def\R{\mathbb R}

\def\C{\mathbb C}

\def\Zero{\mathbb{O}}

\def\Fm{\mathcal{F}}    
\def\FmG{\mathcal{G}}   

\def\HW{HW_3}
\def\gvec{\mathbf{g}}
\def\hvec{\mathbf{h}}

\DeclareMathOperator*{\rank}{rank}

\DeclareMathOperator*{\gl}{GL}
\DeclareMathOperator*{\aut}{Aut}

\DeclareMathOperator*{\modulo}{mod}
\newcommand{\modp}[1]{(\modulo{#1})}

\def\frepName{\rho}                 
\newcommand{\lrep}[1]{\Lambda_{#1}}   
\newcommand{\frep}[1]{\frepName_{#1}}   
\newcommand{\rrep}[1]{R_{#1}}
\newcommand{\sdo}[1]{\left[{#1}\right]} 
\newcommand{\mgZpq}[1][p^q]{\Z_{#1}^{\times}} 
\newcommand{\msg}[2]{\langle{#1}\rangle_{#2}^{\times}}
\newcommand{\msgtwo}[1][p^q]{\msg{2}{#1}}
\newcommand{\msgk}[1][p^q]{\msg{k}{#1}}
\def\gqrest{{\gamma}_q} 

\def\indBij{\phi} 
\def\indBijInv{\phi^{-1}}
\def\funBij{\Phi} 

\newcommand{\accSet}[1]{\Omega_{#1}}

\newcommand{\anMtx}[1][\Fm]{\Theta_{#1}}
\newcommand{\syMtx}[1][\Fm]{\Theta_{#1}^*}
\newcommand{\fmOpMtx}[1][\Fm]{S_{#1}}
\def\Gram{G}
\newcommand{\GmMtx}[1][\Fm]{\Gram_{#1}}

\newcommand{\bindotp}[2]{\dotpRL{#1}{#2}}

\def\convl{*}
\def\sweq{\cong_{\rm{sw}}}

\def\auteq{\cong_{\rm{aut}}}

\newcommand{\GL}[1]{\gl(#1)}
\newcommand{\Aut}[1]{\aut(#1)}

\newcommand{\textSpace}[1]{\text{\;\;\; #1 \;\;\;} }    
\newcommand{\textSpaceHalf}[1]{\text{\;\;\; #1 } }      

\newcommand{\lrphantombrackets}[4]{\left#1\vphantom{#3}\right.#4\left.\vphantom{#3}\right#2}

\definecolor{grn}{rgb}{0.0, 0.5, 0.0}
\definecolor{gray}{rgb}{0.5, 0.5, 0.5}

\newcommand{\lite}[1]{{\color{gray}#1}}

\begin{document}

\begin{abstract}
Binary Parseval frames share many structural properties with real and complex ones. On the other hand, there are subtle differences, for example that the Gramian of a binary Parseval frame is characterized as a symmetric idempotent whose range contains at least one odd vector. 
Here, we study  binary Parseval frames obtained from the orbit of a vector under a group representation, in short, binary Parseval group frames. In this case, the Gramian of the frame is in the algebra generated by the right regular representation. We identify equivalence classes of such Parseval frames with binary functions on the group that satisfy a convolution identity. This allows us to find structural constraints for such frames. We use these constraints to catalogue equivalence
classes of binary Parseval frames obtained from group representations. As an application, we study the performance of binary Parseval  frames generated
with abelian groups for purposes of error correction. We show that $\Z_{p}^q$ is always preferable to $\Z_{p^q}$ when searching for best performing codes
associated with binary Parseval group frames.
\end{abstract}
\keywords{Binary Parseval frames, group representation, group algebra, binary codes; MSC (2010): 42C15, 15A33, 94B05.}

\maketitle

\section{Introduction}
A binary frame is, in short, a finite, spanning family in a  vector space over the Galois field with two elements. Binary frames share many properties
with finite real or complex frames, which have been studied extensively in mathematics and engineering \cite{CK07a,CK07b,FiniteFrames2013}. Because of the spanning property, frames can serve
to expand any given vector in a linear combination of the frame vectors. However, 
in contrast to bases of a vector space, the frame vectors can include linear dependencies. If this is the case, then
the expansion of a vector is no longer uniquely determined. However, for Parseval frames,
there is a standard choice of coefficients appearing in the expansion of a vector that can be
calculated efficiently. When the appropriate definition of a Parseval frame is made, then this holds in the binary as well as the
real and complex setting \cite{BodmannLeRezaTobinTomforde2009}. In the case of real or complex frames, the expansion coefficients are computed by taking inner products with
the frame vectors. The same property for binary frames requires replacing the inner product with the less restrictive concept of a bilinear form \cite{HotovyLarsonScholze2015}, which
can be taken to be the standard dot product.

Next to similarities, binary frames exhibit differences with the theory of real and complex frames. One of the more striking ways in which they differ
is when considering the Gram matrices.
 The Gram matrices of real or complex Parseval frames are characterized as symmetric or Hermitian idempotent matrices. In the binary case, these properties have to be augmented with the condition of at least one non-zero diagonal entry \cite{bgbOdd}. 
 
 In this paper, we continue comparing the structure of binary Parseval frames with their real or complex counterparts.
We specialize to frames obtained from the orbit of a vector under a group representation. There is already a substantial amount of literature on
real and complex group frames \cite{HL,VW,HW,Vale2008,CW,Waldron2013}. 
Here, we study binary Parseval group frames, with special emphasis on the structure of the Gramians associated with them.
Our first main result is that a binary  Parseval frame  is obtained from the action of a group
if and only if the Gramian is in the group algebra. 

Equivalence classes are useful to study essential properties of such frames.
Each set of unitarily equivalent frames can be identified with a corresponding Gramian,
and the Parseval property of such a unitary equivalence class is encoded in its being Hermitian, 
idempotent and having a non-vanishing diagonal element. 

For group frames, the Gram matrix is shown to be a binary linear combination
of elements of the right regular representation, thus associated with a binary function on the group.
This function provides a concise characterization of binary Parseval group frames, thus
allowing us to find structural constraints for such frames. 
We use these constraints to catalogue coarser equivalence classes of binary Parseval frames obtained from group representations.
We specialize further to abelian groups and deduce more specific design constraints.

The results on the structure of binary Parseval frames have significance for the design of error-correcting codes \cite{MWS77}. There have already been a number of works on real or complex frames as codes
\cite{Marshall84,Marshall89,GKK01,RG03,RG04,HP04,BP05,K}.
We continue  efforts of an earlier paper \cite{BodmannCampMahoney2014} to develop results on binary codes from a frame-theoretic perspective. In that paper, methods from graph theory produced results for the design of such codes. Here, we show that when a frame is obtained from the orbit of a vector under a group action, then the choice of the  group can have a significant impact on the coding performance of the resulting frame: Each binary Parseval $\Z_{p^q}$-frame is switching equivalent to a $\Z_p^q$-frame. Thus, when searching for best performers, the group $\Z_p^q$ is a better choice.
We investigate group representations
of $\Z_p^q$ and $\Z_{p^q}$ for small values of $p$ and $q$ and explicitly determine the best performance for $p=q=3$ and $p=5$, $q=3$.

This paper is organized as follows:

Throughout Section \ref{sec:Preliminaries}, we introduce the relevant definitions  of binary group frames, accompanied by illustrating examples.  Fundamental to the theory of real and complex frames is the fact that every group representation that generates a Parseval group frame is unitary. Section \ref{sec:Preliminaries} concludes with the corresponding binary result, along with an explicit formulation of such representations in terms of the frame vectors and the group algebra. 

Section~\ref{sec:gramStructure} is dedicated to the properties of the Gramians of binary Parseval group frames.
Corollary~\ref{cor:gammaFrameIFFGramInAlgebra} summarizes our first main result, that a binary Parseval frame is a group frame if and only if the Gramian is in the group algebra. We follow this with Theorem~\ref{thm:GramCharacterization}, a characterization of the Gramians of binary Parseval group frames in terms of properties of the Gram matrices. This provides the foundation for a characterization of such Gramians in terms of binary functions on the group, Theorem~\ref{thm:coefficientChar}. We devote the remainder of Section~\ref{sec:gramStructure} to developing design constraints specific to binary Parseval group frames induced by abelian groups, leading to a refined characterization of the associated class of functions on the group. This characterization is used in an algorithm that performs an exhaustive search of all binary Parseval group frames for abelian groups of odd order. We walk through an application of this theory in Example~\ref{ex:classifyZ3up2} in order to motivate the ensuing practical guide to classifying binary Parseval frames for such groups.  

In Section~\ref{sec:framesAsCodes}, we apply the structural insights on binary Parseval group frames to the problem of code design. We draw from the results of the exhaustive search for abelian groups of odd order.
We compare the performance of binary Parseval frames generated by the two groups $\Z_p^q$ and $\Z_{p^q}$. 
The remaining examples in the section demonstrate application of these methods, offering comparisons between $\Z_3^3$ and $\Z_{27}$ and between $\Z_5^3$ and $\Z_{125}$.

\section{Preliminaries}\label{sec:Preliminaries}
Unless otherwise noted, the vectors and matrices in this paper are over the field ${\mathbb Z}_2$ containing the  two elements $0$ and $1$.
We write $I_k$ to indicate the $k \times k$ identity matrix over $\Z_2$, occasionally suppressing the subscript when the dimension would not otherwise be noted. We shall refer to the number of nonzero entries of a vector $x\in\Z_2^n$ as the \emph{weight} of $x$ (sometimes written $\normz{x}$), and we say that $x$ is \emph{odd} or \emph{even} if it has an odd or even number of  entries equal to $1$, respectively. These labels extend naturally to the columns and rows of a matrix  viewed as column and row vectors (for example, we may refer to an \emph{odd} or \emph{even column} of a matrix). In keeping with the notation of real or complex frames, we denote the transpose of a binary  matrix $T$ as $T^*$. 
     
Additionally, we may suppress the range of indices on sums and sets for simplicity of notation, as in writing $\sum_j c_j f_j$ for $\sum_{j \in J} c_j f_j$
or $\{f_j\}$ for $\{f_j\}_{j \in J}$ when the index set $J$ is clear from the context.

\subsection{Binary Frames}
Although the dot product as defined below has the appearance of an inner product, it is not positive definite
because taking the dot product of a non-zero even vector with itself gives zero. 
 
\begin{defn}[$\bindotp{\cdot}{\cdot}$, the dot product on $\Z_2^n$]
    We define the bilinear map $\bindotp{\cdot}{\cdot}:\Z_2^n\times\Z_2^n\to\Z_2$, called the \emph{dot product} on $\Z_2^n$, by 
    \[
        \bindotp{
                \begin{bmatrix}
                    a_1\\ \vdots \\a_n
                \end{bmatrix}
                }
                {
                \begin{bmatrix}
                    b_1\\ \vdots \\b_n
                \end{bmatrix}
                }
        :=\sum_{i=1}^n a_i b_i,
    \]
    compactly expressed as $\bindotp{a}{b}=b^*a$ for vectors $a=[a_i]_{i=1}^n,b=[b_i]_{i=1}^n$. 
    Consistent with the language of inner products, we say that two vectors in $\Z_2^n$ are \emph{orthogonal} if their dot product is equal to zero.
\end{defn}

The absence of an inner product motivated the definition of
binary frames 
 in terms of a characterizing feature of finite real and complex frames, that they are a spanning set for the vector space in which they reside \cite{BodmannLeRezaTobinTomforde2009}.

\begin{defn}[Binary frame, binary Parseval frame]
Let $\Fm=\{f_j\}_{j \in J}$ be  a family of vectors in $\Z_2^n$, indexed by a finite set $J$. If $\Fm$ spans $\Z_2^n$, we call $\Fm$ a (binary) frame; if $\Fm$ satisfies the \emph{reconstruction identity}
    \begin{equation}\label{eq:reconstID}
        x=\sum_{j\in J}\bindotp{x}{f_j}f_j\textSpaceHalf{for all }x\in\Z_2^n,
    \end{equation}
    we say that $\Fm$ is a \emph{binary Parseval frame}.
\end{defn}

For other choices of indefinite bilinear form on vector spaces over $\Z_2$ and associated frames, see~\cite{HotovyLarsonScholze2015}. Here, we restrict ourselves to the canonical choice, the dot product.

Since any family of vectors satisfying \eqref{eq:reconstID} necessarily spans $\Z_2^n$, each Parseval frame $\{f_j\}_{j \in J}$ for $\Z_2^n$ is in fact a frame, and the index set necessarily has the size $|J|\geq n$.  
For classification purposes it is useful to introduce equivalence relations among Parseval frames.

\begin{defn}[Unitary binary matrices, unitary equivalence, switching equivalence] We say that a binary $n\times n$ matrix $U$ is \emph{unitary} if $UU^*=U^*U=I_n$.
Given vector families $\Fm:=\{f_j\}_{j \in J}$ and $\Fm':=\{f'_j\}_{j \in J}$ in $\mathbb Z_2^n$, we say that $\Fm$ is \emph{unitarily equivalent}  to $\Fm'$ if there exists a unitary $U\in M_n(\Z_2)$ such that $f_j'=Uf_j$ for all $j\in J$;  we say that $\Fm$ is \emph{switching equivalent} to $\Fm'$ (written $\Fm\sweq\FmG$) if there exists a unitary $U\in M_n(\Z_2)$ and a permutation $\sigma$ on $J$ such that $f_j'=Uf_{\sigma(j)}$ for all $j\in J$.  
\end{defn}

By definition, unitary equivalence is a refinement of switching equivalence. The nature of unitary and permutation matrices makes verifying that these are both equivalence relations a straightforward exercise. Hereafter, we focus on frames indexed by elements of a group. In this context, a restricted version of switching equivalence is useful, which limits the permutations to the subset preserving the group-structure, that is, group automorphisms.

\begin{defn}[Automorphic switching equivalence]
Let $\Gamma$ be a group, we denote the automorphisms of $\Gamma$ by $\aut(\Gamma)$.
Given vector families $\Fm:=\{f_g\}_{g\in\Gamma}$ and $\Fm':=\{f_g'\}_{g\in\Gamma}$ indexed by  $\Gamma$, we say that $\Fm$ and $\Fm'$ are \emph{automorphically switching equivalent} (written $\Fm\auteq\Fm'$) if there exist a unitary $U\in M_n(\Z_2)$ and an automorphism $\sigma \in \aut(\Gamma)$  such that $f_g=Uf_{\sigma(g)}'$ for all $g\in\Gamma$.
\end{defn}

\subsection{Operators associated with a frame}
The following four operators are defined in the same manner as for finite frames over the fields $\R$ and $\C$. In each definition, $\Fm=\{f_j\}_{j \in J}$ is assumed only to be a frame for $\Z_2^n$. 

\begin{defn}[$\anMtx$, the analysis operator,$\syMtx$, the synthesis operator]
We denote the space of $\mathbb Z_2$-valued functions on a set $J$ by $\mathbb Z_2^J$.
    The \emph{analysis operator} for $\Fm$ is the map $\Theta_{\Fm}: \mathbb Z_2^n \to \mathbb Z_2^J$ given by
    $ (\Theta_{\Fm} x)(j) = \langle x, f_j \rangle $.
    The adjoint of $\Theta_\Fm$, also called \emph{synthesis operator}, maps $h \in   \mathbb Z_2^J$ to $\Theta^*_\Fm h = \sum_{j \in J} h(j) f_j$.
    \end{defn}

\begin{defn}[$\fmOpMtx$, the frame operator]
    The \emph{frame operator} for $\Fm$ is the $n\times n$ matrix
    \[
        \fmOpMtx:=\syMtx\anMtx.
    \]
\end{defn}
\begin{rem}
    We note that the reconstruction identity (equation \eqref{eq:reconstID}) may be written as $x=\syMtx\anMtx x$. 
    The reconstruction property of a Parseval frame $\Fm$ for $\Z_2^n$ is equivalent to $\fmOpMtx=I_n$. 
\end{rem}

\begin{defn}[$\GmMtx$, the Gramian]\label{def:gram}
    The \emph{Gramian} for $\Fm$, usually called the \emph{Gram matrix} if $J = \{1, 2, \dots, k\}$, is the linear map $\GmMtx: \Z_2^J \to \Z_2^J$
    \[
        \GmMtx:=\anMtx\syMtx.
    \]
    
Taking $\delta_j(k)=1$ if $k=j$ and $\delta_j(k)=0$ otherwise, $\{\delta_j\}_{j \in J}$ is the standard basis for $\Z_2^J$, and we use matrix notation to write $(G_{\mathcal F})_{i,j} = \langle \Theta_{\mathcal F} \Theta^*_{\mathcal F} \delta_j , \delta_i \rangle = \langle f_j , f_i\rangle$. It follows from the symmetry of $\bindotp{\cdot}{\cdot}$ that  $(\GmMtx)_{i,j}=(\GmMtx)_{j,i}$, and thus $\GmMtx$ is symmetric ($\GmMtx=\GmMtx^*$). Further, if $\Fm$ is a binary Parseval frame, then the Gramian is idempotent:
\begin{equation}
    \GmMtx^2=(\anMtx\syMtx)(\anMtx\syMtx)
        =\anMtx\underbrace{(\syMtx\anMtx)}_{=\fmOpMtx=I_n}\syMtx
        =\anMtx\syMtx
        =\GmMtx    \, .
\end{equation}
\end{defn}

For each of these matrices, we may suppress the subscript if doing so does not cause ambiguity, simply writing $\anMtx[]$ ,$\syMtx[]$, $\fmOpMtx[]$, and $\GmMtx[]$.

\subsection{Group frames for $\Z_2^n$}
Recall that, given a finite group $\Gamma$ and a vector space $V$, a \emph{representation} of $\Gamma$ on $V$ is a group homomorphism
\[    
        \frep{}: \Gamma\to \GL V.
\]
In such a case, we say that $\Gamma$ \emph{acts on} $V$ by $\frep{}$, and for any group element $g$ we may interchangeably write $\frep{}(g)$ as $\frep{g}$. We shall refer to the elements of $\{\frep{g}\}_{g\in\Gamma}$ as the \emph{matrices of the representation} $\frep{}$, or simply as representation matrices.  Further, if each of the matrices $\frep{g}$ is unitary, then we call the representation itself unitary. 

In the context of complex Hilbert spaces, given a finite group $\Gamma$, a \emph{group frame} generated by $\Gamma$ is any frame $\{f_g\}_{g\in\Gamma}$ that satisfies $\frep{g}f_h=f_{gh}$ for all $g,h\in\Gamma$, for some representation $\frep{}$ of $\Gamma$.  If that representation is unitary, the frame is the orbit of a single vector \cite{Waldron2013}; it is this idea of a group generating a frame from a single vector that provides the basis of our definition.  

\begin{defn}[Binary Parseval group frame, $\Gamma$-frame]
Given a natural number $n$ and a group $\Gamma$ acting on the vector space $\Z_2^n$ by a representation $\frep{}$, let $\Fm:=\{\frep{g}f \}_{g\in\Gamma}$ denote the orbit of a vector $f\in \Z_2^n$ under $\frep{}$. If $\Fm$ spans $\Z_2^n$, then it is a frame which we call a \emph{binary group frame}.  If $\Fm$ is a Parseval frame, we say that it is a \emph{binary Parseval group frame}. 
For a given group $\Gamma$, we abbreviate the description ``group frame generated by $\Gamma$'' as $\Gamma$-frame \cite{Vale2008}. We shall index frame vectors by their inducing group elements, so that $f_e:=f$ and $f_g:=\frep{}(g)f=\frep{g}f$ for $g\in\Gamma$.
\end{defn}
We begin with examples of frames generated with groups of size $27$ acting on $\Z_2^9$. These examples show that depending on the
choice of $f_e$, a unitary group  representation may lead to an orbit that is a Parseval frame or just a frame.

\begin{exs}[Two binary cyclic frames] Let $\Gamma = \Z_{27}$ be the group of integers $\{0,1, \dots, 26\}$ with addition modulo $27$.
 Let $S_9$ be the  cyclic shift on $\Z_2^9$, so for each canonical basis vector $e_i$ with $i \le 8$,  $S_9 e_i = e_{i+1}$ and $S_9e_9 = e_1$.
 Since $S_9$ is a permutation matrix, the map $\rho: i \mapsto S_9^i$ is a homomorphism
 from $ \Gamma$ to $GL(\Z_2^9)$. 
 Choosing $f^*_e = [1\, 0\, 1\, 1\, 1\, 1\, 1\, 1\, 0]$ gives that $\{f_j\}_{j \in \Gamma}$ spans $\Z_2^9$, 
 but $\Theta_\Fm^* \Theta_\Fm \ne I_9$, so $\Fm$ is a frame but not Parseval.
   
 Moreover, choosing $f_e=e_1$ shows that $\{f_i\}_{i \in \Gamma}$ with $f_i = S_9^i e_1 = e_{1 + i\modp{9}}$ 
 and $e_0 \equiv e_9$ repeats the sequence of 
 canonical basis vectors
 three times. Consequently, the synthesis operator is $\Theta^*_{\mathcal F} = \left[ I_9\, I_9\,  I_9 \right]$
 and $\Theta_\Fm^* \Theta_{\Fm} = I_9$, so $\Fm$ is Parseval.
 \end{exs}

An exhaustive search of all Parseval frames obtained from group orbits under the action of $\Z_{27}$  on $\Z_2^9$
reveals that up to unitary equivalence, the second example is the only case of a Parseval $\Z_{27}$-frame for $\Z_2^9$. 
Such an exhaustive search is made feasible by methods developed in Section~\ref{sec:methodsThy}.
In a preceding paper, the linear dependence among repeated frame vectors, as exhibited in the frame having synthesis operator $\left[ I_9\, I_9\,  I_9 \right]$, has been called \emph{trivial redundancy} \cite{BodmannLeRezaTobinTomforde2009};
In the next example, we show that another 
group of the same size generates frames as well as Parseval frames without the occurrence of repeated vectors in either case.
 


\begin{ex}[Two binary Gabor frames]\label{ex:GaborFrames}
Let $a,b\in \GL{\Z_3^3}$ and  the suggestively named $\rho_a,\rho_b\in \GL{\Z_2^9}$ be defined by 
\[
a:=\left[\begin{smallmatrix}
1&0&0\\0&1&1\\0&0&1
\end{smallmatrix}\right] \, , \quad
b:=\left[\begin{smallmatrix}
1&1&0\\0&1&0\\0&0&1
\end{smallmatrix}\right] \, , \quad
\rho_a:=\left[\begin{smallmatrix}
I_3&\Zero&\Zero\\\Zero&X&\Zero\\\Zero&\Zero&Y
\end{smallmatrix}\right]
\, , \quad
\rho_b:=\left[\begin{smallmatrix}
\Zero&I_3&\Zero\\\Zero&\Zero&I_3\\I_3&\Zero&\Zero
\end{smallmatrix}\right],
\]
where $\Zero$ is the $3\times3$ matrix of zeros, $X=\left[\begin{smallmatrix}
0&0&1\\1&0&0\\0&1&0
\end{smallmatrix}\right]$, and $Y=X^2$.

The group generated by $a$ and $b$ under matrix multiplication is the nonabelian finite \emph{Heisenberg-Weyl} group modulo 3, denoted $\HW$.
From the fact that products of powers of $a$ and $b$ give all upper triangular ternary matrices whose diagonal entries are fixed at $1$, one can deduce that this group has order 27, see also \cite{Schwinger1960}. 
The matrices $\rho_a$ and $\rho_b$ generate a group isomorphic to $\HW$ and have been chosen such that setting $\rho(a):=\rho_a$ and $\rho(b):=\rho_b$ extends to an isomorphism $\rho:\HW\to\GL{\Z_2^9}$. For compactness of notation, we designate a third group element\footnote{Note that $\{e,c,c^2\}$ is the center of $\HW$.} $c$ and corresponding $\rho_c\in \GL{\Z_2^9}$
\[
c:=\left[\begin{smallmatrix}
1&0&1\\0&1&0\\0&0&1
\end{smallmatrix}\right]
\textSpace{and}
\rho_c:=\rho(c)=\left[\begin{smallmatrix}
X&\Zero&\Zero\\\Zero&X&\Zero\\\Zero&\Zero&X
\end{smallmatrix}\right],
\]
and order the elements of $\HW$ in the sequence
\[\begin{array}{ccc}
\text{\footnotesize$e$, $a$, $a^2$, $b$, $ab$, $a^2b$, $b^2$, $ab^2$, $a^2b^2$,}\\
\text{\footnotesize$c$, $ac$, $a^2c$, $bc$, $abc$, $a^2bc$, $b^2c$, $ab^2c$, $a^2b^2c$,}\\
\text{\footnotesize$c^2$, $ac^2$, $a^2c^2$, $bc^2$, $abc^2$, $a^2bc^2$, $b^2c^2$, $ab^2c^2$, $a^2b^2c^2$.}
\end{array}\]

Choosing $f_e = [1\, 1\, 0\, 1\, 0\, 0\, 0\, 0\, 0]^*$ and $ f'_e = [1\, 1\, 0\, 1\, 1\, 0\, 1\, 0\, 0]^*$ induces $\HW$-frames $\{f_e\}$ and $\{f_e'\}$ whose synthesis operators $\Theta_1^*:=[f_{e}|f_a|f_{a^2}|\cdots|f_{a^2b^2c^2}]$ and $\Theta_2^*:=[f'_{e}|f'_a|f'_{a^2}|\cdots|f'_{a^2b^2c^2}]$ are given in Figure \ref{fig:GaborFrames}.  One may verify that $\{f'_g\}_{g\in\HW}$ is the only Parseval frame of the pair by calculating the corresponding frame operators.
\begin{figure}[h]
    \centering
\[
\begin{array}{cc}\underset{{\large\mbox{$\Theta_1^*$}}}{
 \left[\begin{smallmatrix}
1 & 1 & 1 & 1 & 1 & 1 & \lite{0} & \lite{0} & \lite{0} & \lite{0} & \lite{0} & \lite{0} & \lite{0} & \lite{0} & \lite{0} & \lite{0} & \lite{0} & \lite{0} & 1 & 1 & 1 & \lite{0} & \lite{0} & \lite{0} & \lite{0} & \lite{0} & \lite{0}\\
1 & 1 & 1 & \lite{0} & \lite{0} & \lite{0} & \lite{0} & \lite{0} & \lite{0} & 1 & 1 & 1 & 1 & 1 & 1 & \lite{0} & \lite{0} & \lite{0} & \lite{0} & \lite{0} & \lite{0} & \lite{0} & \lite{0} & \lite{0} & \lite{0} & \lite{0} & \lite{0}\\
\lite{0} & \lite{0} & \lite{0} & \lite{0} & \lite{0} & \lite{0} & \lite{0} & \lite{0} & \lite{0} & 1 & 1 & 1 & \lite{0} & \lite{0} & \lite{0} & \lite{0} & \lite{0} & \lite{0} & 1 & 1 & 1 & 1 & 1 & 1 & \lite{0} & \lite{0} & \lite{0}\\
1 & \lite{0} & \lite{0} & \lite{0} & \lite{0} & \lite{0} & 1 & \lite{0} & 1 & \lite{0} & \lite{0} & 1 & \lite{0} & \lite{0} & \lite{0} & \lite{0} & 1 & 1 & \lite{0} & 1 & \lite{0} & \lite{0} & \lite{0} & \lite{0} & 1 & 1 & \lite{0}\\
\lite{0} & 1 & \lite{0} & \lite{0} & \lite{0} & \lite{0} & 1 & 1 & \lite{0} & 1 & \lite{0} & \lite{0} & \lite{0} & \lite{0} & \lite{0} & 1 & \lite{0} & 1 & \lite{0} & \lite{0} & 1 & \lite{0} & \lite{0} & \lite{0} & \lite{0} & 1 & 1\\
\lite{0} & \lite{0} & 1 & \lite{0} & \lite{0} & \lite{0} & \lite{0} & 1 & 1 & \lite{0} & 1 & \lite{0} & \lite{0} & \lite{0} & \lite{0} & 1 & 1 & \lite{0} & 1 & \lite{0} & \lite{0} & \lite{0} & \lite{0} & \lite{0} & 1 & \lite{0} & 1\\
\lite{0} & \lite{0} & \lite{0} & 1 & 1 & \lite{0} & 1 & \lite{0} & \lite{0} & \lite{0} & \lite{0} & \lite{0} & \lite{0} & 1 & 1 & \lite{0} & 1 & \lite{0} & \lite{0} & \lite{0} & \lite{0} & 1 & \lite{0} & 1 & \lite{0} & \lite{0} & 1\\
\lite{0} & \lite{0} & \lite{0} & 1 & \lite{0} & 1 & \lite{0} & \lite{0} & 1 & \lite{0} & \lite{0} & \lite{0} & 1 & 1 & \lite{0} & 1 & \lite{0} & \lite{0} & \lite{0} & \lite{0} & \lite{0} & \lite{0} & 1 & 1 & \lite{0} & 1 & \lite{0}\\
\lite{0} & \lite{0} & \lite{0} & \lite{0} & 1 & 1 & \lite{0} & 1 & \lite{0} & \lite{0} & \lite{0} & \lite{0} & 1 & \lite{0} & 1 & \lite{0} & \lite{0} & 1 & \lite{0} & \lite{0} & \lite{0} & 1 & 1 & \lite{0} & 1 & \lite{0} & \lite{0}\\[0.2em]
\end{smallmatrix}
\right]}&\underset{{\large\mbox{ $\Theta_2^*$}}}{
 \left[\begin{smallmatrix}
1 & 1 & 1 & 1 & 1 & 1 & 1 & 1 & 1 & \lite{0} & \lite{0} & \lite{0} & \lite{0} & \lite{0} & \lite{0} & \lite{0} & \lite{0} & \lite{0} & 1 & 1 & 1 & 1 & 1 & 1 & \lite{0} & \lite{0} & \lite{0}\\
1 & 1 & 1 & 1 & 1 & 1 & \lite{0} & \lite{0} & \lite{0} & 1 & 1 & 1 & 1 & 1 & 1 & 1 & 1 & 1 & \lite{0} & \lite{0} & \lite{0} & \lite{0} & \lite{0} & \lite{0} & \lite{0} & \lite{0} & \lite{0}\\
\lite{0} & \lite{0} & \lite{0} & \lite{0} & \lite{0} & \lite{0} & \lite{0} & \lite{0} & \lite{0} & 1 & 1 & 1 & 1 & 1 & 1 & \lite{0} & \lite{0} & \lite{0} & 1 & 1 & 1 & 1 & 1 & 1 & 1 & 1 & 1\\
1 & \lite{0} & 1 & 1 & \lite{0} & \lite{0} & 1 & \lite{0} & 1 & \lite{0} & 1 & 1 & \lite{0} & \lite{0} & 1 & \lite{0} & 1 & 1 & 1 & 1 & \lite{0} & \lite{0} & 1 & \lite{0} & 1 & 1 & \lite{0}\\
1 & 1 & \lite{0} & \lite{0} & 1 & \lite{0} & 1 & 1 & \lite{0} & 1 & \lite{0} & 1 & 1 & \lite{0} & \lite{0} & 1 & \lite{0} & 1 & \lite{0} & 1 & 1 & \lite{0} & \lite{0} & 1 & \lite{0} & 1 & 1\\
\lite{0} & 1 & 1 & \lite{0} & \lite{0} & 1 & \lite{0} & 1 & 1 & 1 & 1 & \lite{0} & \lite{0} & 1 & \lite{0} & 1 & 1 & \lite{0} & 1 & \lite{0} & 1 & 1 & \lite{0} & \lite{0} & 1 & \lite{0} & 1\\
1 & \lite{0} & \lite{0} & 1 & 1 & \lite{0} & 1 & 1 & \lite{0} & \lite{0} & 1 & \lite{0} & \lite{0} & 1 & 1 & \lite{0} & 1 & 1 & \lite{0} & \lite{0} & 1 & 1 & \lite{0} & 1 & 1 & \lite{0} & 1\\
\lite{0} & \lite{0} & 1 & 1 & \lite{0} & 1 & 1 & \lite{0} & 1 & 1 & \lite{0} & \lite{0} & 1 & 1 & \lite{0} & 1 & 1 & \lite{0} & \lite{0} & 1 & \lite{0} & \lite{0} & 1 & 1 & \lite{0} & 1 & 1\\
\lite{0} & 1 & \lite{0} & \lite{0} & 1 & 1 & \lite{0} & 1 & 1 & \lite{0} & \lite{0} & 1 & 1 & \lite{0} & 1 & 1 & \lite{0} & 1 & 1 & \lite{0} & \lite{0} & 1 & 1 & \lite{0} & 1 & 1 & \lite{0}\\[0.2em]
\end{smallmatrix}
\right]}\end{array}.
\]
    \caption{Synthesis operators of two binary Gabor frames (see Ex. \ref{ex:GaborFrames})}
    \label{fig:GaborFrames}
\end{figure}
\end{ex}

\subsection{Regular representations and group frames}\label{sec:regReps}
Constructing a faithful unitary representation of a finite group $\Gamma$ is always possible; if $\Gamma$ has order $k$ and given a $k$-dimensional vector space $V$, one can find a collection of $k\times k$ permutation matrices $\{P_g\}_{g\in\Gamma}\subset \GL{V}$ that form a group isomorphic to $\Gamma$. This is just a result of Cayley's theorem for groups based on the left or right regular representation, as given below. The challenge is to find vector spaces of smaller dimension carrying
a unitary representation and vectors whose orbits under the group action form a Parseval frame.

With this in mind, we recall that the \emph{regular representations} of a finite group $\Gamma$ over a field $\mathbb{K}$ act on $\mathbb{K}^\Gamma$, the vector space of 
 $\bb{K}$-valued functions on $\Gamma$; the \emph{left} regular representation $\lrep{}=\{\lrep{g}\}_{g \in \Gamma}$ and \emph{right} regular representation $\rrep{}
 = \{\rrep{g}\}_{g \in \Gamma}$ 
 act on $\varphi: \Gamma \to \mathbb K$ according to
 \[
\lrep{g}\varphi: h\mapsto \varphi(g^{-1}h)
    \textSpace{and}
    \rrep{g}\varphi: h\mapsto \varphi(h g)
\]
and hence define group isomorphisms. By associativity, $\lrep{g}\rrep{h}\varphi$ and $\rrep{h}\lrep{g}\varphi$ are well defined, and commutativity among operators of the regular representations follows from the chain of equalities
\[
\big(\lrep{g}\rrep{h}\varphi\big)(x)
    =\big(\rrep{h}\varphi\big)(g^{-1}x)
    =\varphi(g^{-1}xh)
    =\big(\lrep{g}\varphi\big)(xh)
    = \big(\rrep{h}\lrep{g}\varphi\big)(x)
\]
which holds for all $g,h,x\in\Gamma$ and $\varphi \in {\mathbb K}^\Gamma$.

\begin{rem}\label{rem:repEntries}
For future use, we note that for each $g\in\Gamma$, the nonzero entries of the permutation matrix associated with $\lrep{g}$ and the canonical basis are indexed by the set $\{(gh,h)\colon h\in\Gamma\}$, and $\rrep{g}$ is nonzero exactly on index set $\{(hg^{-1},h)\colon h\in\Gamma\}$. Written in terms of the Kronecker delta,
\[
(\lrep{g})_{\alpha,\beta}=\delta_{\alpha\beta^{-1}}^g \textSpace{and}(\rrep{g})_{\alpha,\beta}=\delta_{\alpha^{-1}\beta}^g,
\textSpace{ where }
\delta_i^j:=\begin{cases}1 \text{ if } i=j\\ 0 \text{ if } i\neq j   \end{cases}.
\]
\end{rem}
 
We close the section by showing that, as in the real or complex case \cite{Waldron2013}, the group representations which generate binary Parseval group frames are unitary.
\begin{prop}[Binary Parseval group frames are generated by unitary representations]\label{prop:rhoUnitary}
Given a finite group $\Gamma$, let $\Fm$ be a binary $\Gamma$-frame generated by a group representation $\frep{}$. If $\Fm$ is Parseval
with analysis operator $\Theta$, then $\frep{}$ is a unitary representation with matrices explicitly given by $\frep{g}=\Theta^*\lrep{g}\Theta$ for each $g\in\Gamma$.
\end{prop}

\begin{proof}
Let $\Gamma$, $\Fm$ and $\Theta$ be as in the hypothesis.  For $g,h\in\Gamma$ and $x \in \mathbb Z_2^n$,
\[
 ( \Lambda_g \Theta x)(h) = \Theta x (g^{-1} h) = \langle x, \rho_{g^{-1}} f_h \rangle = \langle \rho_{g^{-1}}^* x, f_h  \rangle = ( \Theta \rho_{g^{-1}}^* x) (h),
\] 
so the following diagram commutes:
\begin{displaymath}
\xymatrix{
 \Z_2[\Gamma] \ar[r]^{\lrep{g}} & \Z_2[\Gamma] \\
 \Z_2^n \ar[u]^{\anMtx[]}\ar[r]^{(\frep{g^{-1}})^*} & \Z_2^n\ar[u]_{\anMtx[]}
}
\end{displaymath}




By the Parseval property and the intertwining relationship, $\rho_{g^{-1}}^* = \Theta^* \Theta \rho_{g^{-1}}^*=\Theta^* \Lambda_g \Theta$.
Replacing $g$ with $g^{-1}$ and taking the transpose, we then get $\rho_{g} = \Theta^* \Lambda^*_{g^{-1}} \Theta$. 
Next, the unitarity of $\Lambda_g$ gives the claimed expression $\rho_{g} = \Theta^* \Lambda_{g} \Theta$.
It follows that
\begin{align*}
    \rho_{g^{-1}}^* 
        &= \hphantom{(}\Theta^* \Lambda_g \Theta\\
        &= (\Theta^* \Lambda_{g^{-1}} \Theta)^*\\
        &=(\rho_g^*)^* 
            \equiv\rho_g. 
\end{align*}
We conclude that each $\rho_g$ is unitary.
\end{proof}

\begin{ex}[A binary Parseval $\Z_3^2$-frame]\label{ex:z3up2Frame} The family of vectors and matrix 
\[
\Fm=
\lrphantombrackets{\{}{\}}{\begin{smallmatrix}
0\\0\\0\\0\\0\\\\0
\end{smallmatrix}}{\!\!
\underset{f_{\!\tiny\nCkGen{0\\0}}}{
\left[\begin{smallmatrix}
0\\1\\0\\0\\0\end{smallmatrix}
\right]},\underset{f_{\!\tiny\nCkGen{0\\1}}}{
\left[\begin{smallmatrix}
0\\0\\1\\0\\0\end{smallmatrix}
\right]},\underset{f_{\!\tiny\nCkGen{0\\2}}}{
\left[\begin{smallmatrix}
0\\0\\0\\1\\0\end{smallmatrix}
\right]},\underset{f_{\!\tiny\nCkGen{1\\0}}}{
\left[\begin{smallmatrix}
0\\1\\1\\0\\1\end{smallmatrix}
\right]},\underset{f_{\!\tiny\nCkGen{1\\1}}}{
\left[\begin{smallmatrix}
0\\0\\1\\1\\1\end{smallmatrix}
\right]},\underset{f_{\!\tiny\nCkGen{1\\2}}}{
\left[\begin{smallmatrix}
0\\1\\0\\1\\1\end{smallmatrix}
\right]},\underset{f_{\!\tiny\nCkGen{2\\0}}}{
\left[\begin{smallmatrix}
1\\1\\0\\1\\0\end{smallmatrix}
\right]},\underset{f_{\!\tiny\nCkGen{2\\1}}}{
\left[\begin{smallmatrix}
1\\1\\1\\0\\0\end{smallmatrix}
\right]},\underset{f_{\!\tiny\nCkGen{2\\2}}}{
\left[\begin{smallmatrix}
1\\0\\1\\1\\0\end{smallmatrix}
\right]}
\!\!}
\textSpace{and}
G=\left[\begin{smallmatrix}
      1  & \lite{0} & \lite{0} &       1  & \lite{0} &       1  &       1  &       1  & \lite{0}\\
\lite{0} &       1  & \lite{0} &       1  &       1  & \lite{0} & \lite{0} &       1  &       1 \\
\lite{0} & \lite{0} &       1  & \lite{0} &       1  &       1  &       1  & \lite{0} &       1 \\
      1  &       1  & \lite{0} &       1  & \lite{0} & \lite{0} &       1  & \lite{0} &       1 \\
\lite{0} &       1  &       1  & \lite{0} &       1  & \lite{0} &       1  &       1  & \lite{0}\\
      1  & \lite{0} &       1  & \lite{0} & \lite{0} &       1  & \lite{0} &       1  &       1 \\
      1  & \lite{0} &       1  &       1  &       1  & \lite{0} &       1  & \lite{0} & \lite{0}\\
      1  &       1  & \lite{0} & \lite{0} &       1  &       1  & \lite{0} &       1  & \lite{0}\\
\lite{0} &       1  &       1  &       1  & \lite{0} &       1  & \lite{0} & \lite{0} &       1 
\end{smallmatrix}
\right]
\]
are a binary Parseval $\Z_3^2$ frame and its Gramian.
Denoting the left regular representation of $\Z_3$ as $\frep{}'$ with $\frep{1}':=\left[\begin{smallmatrix}
0&0&1\\1&0&0\\0&1&0
\end{smallmatrix}\right]$, the left regular representation matrices of $\Z_3^2$ are defined by the Kronecker products $\lrep{\nCksm{i}{j}}=\frep{i}'\otimes\frep{j}'$,
with $\rho'_i = (\rho'_1)^i$ for $i \in {\mathbb Z}_2$. The corresponding matrices $\frep{}\nCksm{i}{j}:=\syMtx[]\lrep{\nCksm{i}{j}}\anMtx[]$ provide a representation of the group $\Z_3^2$ on the vector space $\Z_2^5$. 

\[
\begin{array}{ccc}\rho\colVecsm{0\\0}=\underset{}{
\left[\begin{smallmatrix}
      1  & \lite{0} & \lite{0} & \lite{0} & \lite{0}\\
\lite{0} &       1  & \lite{0} & \lite{0} & \lite{0}\\
\lite{0} & \lite{0} &       1  & \lite{0} & \lite{0}\\
\lite{0} & \lite{0} & \lite{0} &       1  & \lite{0}\\
\lite{0} & \lite{0} & \lite{0} & \lite{0} &       1 
\end{smallmatrix}
\right]},&\rho\colVecsm{0\\1}=\underset{}{
\left[\begin{smallmatrix}
      1  & \lite{0} & \lite{0} & \lite{0} & \lite{0}\\
\lite{0} & \lite{0} & \lite{0} &       1  & \lite{0}\\
\lite{0} &       1  & \lite{0} & \lite{0} & \lite{0}\\
\lite{0} & \lite{0} &       1  & \lite{0} & \lite{0}\\
\lite{0} & \lite{0} & \lite{0} & \lite{0} &       1 
\end{smallmatrix}
\right]},&\rho\colVecsm{0\\2}=\underset{}{
\left[\begin{smallmatrix}
      1  & \lite{0} & \lite{0} & \lite{0} & \lite{0}\\
\lite{0} & \lite{0} &       1  & \lite{0} & \lite{0}\\
\lite{0} & \lite{0} & \lite{0} &       1  & \lite{0}\\
\lite{0} &       1  & \lite{0} & \lite{0} & \lite{0}\\
\lite{0} & \lite{0} & \lite{0} & \lite{0} &       1 
\end{smallmatrix}
\right]},\\\rho\colVecsm{1\\0}=\underset{}{
\left[\begin{smallmatrix}
\lite{0} & \lite{0} & \lite{0} & \lite{0} &       1 \\
      1  &       1  & \lite{0} &       1  & \lite{0}\\
      1  &       1  &       1  & \lite{0} & \lite{0}\\
      1  & \lite{0} &       1  &       1  & \lite{0}\\
\lite{0} &       1  &       1  &       1  & \lite{0}
\end{smallmatrix}
\right]},&\rho\colVecsm{1\\1}=\underset{}{
\left[\begin{smallmatrix}
\lite{0} & \lite{0} & \lite{0} & \lite{0} &       1 \\
      1  & \lite{0} &       1  &       1  & \lite{0}\\
      1  &       1  & \lite{0} &       1  & \lite{0}\\
      1  &       1  &       1  & \lite{0} & \lite{0}\\
\lite{0} &       1  &       1  &       1  & \lite{0}
\end{smallmatrix}
\right]},&\rho\colVecsm{1\\2}=\underset{}{
\left[\begin{smallmatrix}
\lite{0} & \lite{0} & \lite{0} & \lite{0} &       1 \\
      1  &       1  &       1  & \lite{0} & \lite{0}\\
      1  & \lite{0} &       1  &       1  & \lite{0}\\
      1  &       1  & \lite{0} &       1  & \lite{0}\\
\lite{0} &       1  &       1  &       1  & \lite{0}
\end{smallmatrix}
\right]},\\\rho\colVecsm{2\\0}=\underset{}{
\left[\begin{smallmatrix}
\lite{0} &       1  &       1  &       1  & \lite{0}\\
\lite{0} &       1  &       1  & \lite{0} &       1 \\
\lite{0} & \lite{0} &       1  &       1  &       1 \\
\lite{0} &       1  & \lite{0} &       1  &       1 \\
      1  & \lite{0} & \lite{0} & \lite{0} & \lite{0}
\end{smallmatrix}
\right]},&\rho\colVecsm{2\\1}=\underset{}{
\left[\begin{smallmatrix}
\lite{0} &       1  &       1  &       1  & \lite{0}\\
\lite{0} &       1  & \lite{0} &       1  &       1 \\
\lite{0} &       1  &       1  & \lite{0} &       1 \\
\lite{0} & \lite{0} &       1  &       1  &       1 \\
      1  & \lite{0} & \lite{0} & \lite{0} & \lite{0}
\end{smallmatrix}
\right]},&\rho\colVecsm{2\\2}=\underset{}{
\left[\begin{smallmatrix}
\lite{0} &       1  &       1  &       1  & \lite{0}\\
\lite{0} & \lite{0} &       1  &       1  &       1 \\
\lite{0} &       1  & \lite{0} &       1  &       1 \\
\lite{0} &       1  &       1  & \lite{0} &       1 \\
      1  & \lite{0} & \lite{0} & \lite{0} & \lite{0}
\end{smallmatrix}
\right]}.\end{array}
\]

One may verify that the induced map $\frep{}:\Z_3^2\to\GL{\Z_2^5}$ is a unitary representation and that $\Fm$ is in fact the orbit of $f_{\!\scriptsize\colVecsm{0\\0}}$ under $\frep{}$.
\end{ex}

\section{The structure of the Gramian of a binary Parseval group frame}\label{sec:gramStructure}
The Gramian captures geometric information about the structure of the associated frame, since it records every pairwise dot product among frame vectors.  If the frame
is a group frame, then it also reflects the group structure. As shown below, the Gramian of a binary Parseval group frame is an element of the algebra
generated by the right regular representation.

\begin{thm}[Gramians of binary Parseval group frames as elements of the group algebra]\label{thm:gramInAlgebra}
Let \linebreak $\Gamma$ be a finite group with right regular representation
$\{\rrep{g}\}_{g\in\Gamma}$ and associated group algebra
$\Z_2[\{\rrep{g}\}]$, and suppose $G$ is the Gramian of binary Parseval
$\Gamma$-frame $\Fm:=\{f_g\}_{g\in\Gamma}$, then the Gram matrix is in
the group algebra. More explicitly, $G$ is given by
\begin{equation}
  G=\sum_{g\in\Gamma}\eta(g)\rrep{g}
\end{equation}
where the function $\eta:\Gamma\to\Z_2$ is defined by $\eta(g):=\bindotp{f_g}{f_e}$.
\end{thm} 
\begin{proof}
Let $\frep{}$ be the frame-generating group representation, which by Proposition \ref{prop:rhoUnitary} is a unitary representation of $\Gamma$.
Consider $H = \sum_{g \in \Gamma} \eta(g) R_g$ with $\eta$ as in the statement of the theorem.
We compute for $a,b \in \Gamma$ the value 
\begin{align*} H_{a,b} & = \sum_{g \in \Gamma} \eta(g) (R_g)_{a,b} \\
& = \sum_{g \in \Gamma } \eta(g) \delta_{a^{-1} b}^g = \eta(a^{-1} b)\\
 & = \langle f_{a^{-1} b}, f_e \rangle \\
 & = \langle \rho_b f_e, \rho_{a^{-1}}^* f_e \rangle \\
 & = \langle f_b, f_a \rangle = G_{a,b} \, .
 \end{align*}
 In the last identity, we have used the unitarity, $\rho_{a^{-1}} = \rho_a^*$.
 \end{proof}
 \begin{thm}[Gramians of binary Parseval frames in a group algebra imply group frame structure]
 Let $\Gamma$ be a finite group with regular representations $\lrep{}$ and $\rrep{}$, and suppose
$\Fm$ is a binary Parseval frame indexed by $g \in \Gamma$ with Gramian $G$ and analysis operator $\anMtx[]$. If $G$ is in the group algebra $\Z_2[\{\rrep{g}\}]$, then
$\rho_g := \syMtx[] \lrep{g} \anMtx[]$ defines a unitary representation of $\Gamma$ and $\{\frep{g}f_e\}_{g \in \Gamma}=\{f_g\}_{g \in \Gamma}$ is a $\Gamma$-frame.
 \end{thm}
 \begin{proof}
 Assume that $G=\anMtx[]\syMtx[] \in \mathbb Z_2 [ \{R_g \}]$.
Since $\Lambda_g $ and $R_{g'}$ commute for  each $g, g' \in \Gamma$, so do $\Lambda_g$ and $G$.  From $\Theta^* \Theta \Theta^* = \Theta^*$, then, we have $\Theta^* \Lambda_g \Theta \Theta^* \Lambda_h \Theta = \Theta^* \Lambda_{gh} \Theta$ for each $g, h \in \Gamma$, and since $(\Theta^* \Lambda_g \Theta)^* = \Theta^* \Lambda_{g^{-1}} \Theta$, it follows that
$\rho_g = \Theta^* \Lambda_g \Theta$ defines a unitary representation of $\Gamma$. Using these properties for $\rho_g$ then
shows that $$\rho_g f_e = \Theta^* \Lambda_g \Theta \Theta^* \delta_e = \Theta^* \Theta \Theta^* \Lambda_g \delta_e = \Theta^* \delta_g = f_g \, ,$$
so the frame vectors are obtained from the orbit under the unitaries $\{\rho_g\}_{g \in \Gamma}$. 
 \end{proof}
 
 We summarize the preceding two theorems in a characterization of
 binary Parseval group frames.
 
 \begin{cor}[Characterization of binary Parseval group frames in terms of Gramians]\label{cor:gammaFrameIFFGramInAlgebra}
 Let $\Gamma$ be a finite group with right regular representation  $\{\rrep{g}\}_{g\in\Gamma}$.
 A binary Parseval frame $\mathcal F$ indexed by $\Gamma$ is a $\Gamma$-frame if and only if
 its Gramian is in the algebra $\Z_2[\{\rrep{g}\}_{g\in\Gamma}]$.
 \end{cor}

\subsection{Characterizing the structure of the Gramian}
In order to facilitate a catalogue of binary Parseval group frames, we identify necessary and sufficient conditions for their Gramians.

In the real or complex case, each symmetric idempotent matrix is the Gram matrix of a Parseval frame.
In the binary case, \cite[Theorem 4.1]{bgbOdd} characterizes Parseval frames with the additional requirement that
at least one row or column vector is odd. This condition is equivalent to the condition that the Gramian has at least one odd vector in its range, since the span of the column vectors of a matrix forms the range of the matrix; we use these statements interchangeably throughout this paper. 
This condition is \emph{also} equivalent to that of having at least one nonzero entry on the diagonal, since the idempotence and symmetry of a Gramian $G$ induce the identity between the dot product of a vector $G\delta_g$ with itself and the corresponding diagonal entry of the Gramian,  $ \bindotp{G\delta_g}{G\delta_g}=G_{g,g}$ for all $g\in\Gamma$. 

We next combine the results we obtained so far with the characterization of the Gramians of binary Parseval frames to characterize Gramians that belong to binary Parseval group frames.
 \begin{thm}[The structure of Gramians of binary Parseval group frames]\label{thm:GramCharacterization}
 Given a finite group $\Gamma$ with right regular representation $\rrep{}$, a map $G:\Z_2^\Gamma\to\Z_2^\Gamma$ is the Gramian of binary Parseval $\Gamma$-frame
 if and only if $G$ is symmetric and idempotent,
$G\in\Z_2[\{\rrep{g}\}]$ and the range of $G$ contains an odd vector.
\end{thm}

\begin{proof}
    As noted above, \cite[Theorem 4.1]{bgbOdd} characterizes the Gram matrices of binary Parseval frames as symmetric, idempotent matrices having at least one odd column. Thus, given a finite group $\Gamma$, the characterization in the current theorem reduces to Corollary \ref{cor:gammaFrameIFFGramInAlgebra}, and is thereby proven.
\end{proof}

\subsection{Additional properties of the Gramian}
Since regular representation matrices are permutation matrices, a consequence of Theorem \ref{thm:gramInAlgebra} is that each of the rows of the Gramian of a binary Parseval group frame has the same weight. Thus, if a Gramian is assumed to be that of a binary Parseval group frame, the condition that one column is odd is equivalent to the condition that \emph{every} column is odd, which equates to the condition that every diagonal entry is a $1$, or even simply that $\eta(e)=1$. Continuing under the assumption that the Gramian may be written as $G=\Sigma_g \eta(g)\rrep{g}$, the quantity of $1$'s in a column is the quantity of elements $g\in\Gamma$ such that $\eta(g)=1$; it follows that $G$ has an odd column if and only if the sum $\Sigma_g \eta(g)=1$.

Now, suppose $\Gamma$ is a finite group of order $k$ and that we wish to exhaustively search for $\Gamma$-frames. The characterization in Theorem \ref{thm:GramCharacterization} tells us that the candidate set of Gramians is a subset of 
\begin{equation}\label{set:etaProps}
\left\{H= \sum_{g\in\Gamma}\eta(g)\rrep{g}\;
\begin{array}{|l}
\eta(e)=1\\
\eta(g)=\eta(g^{-1}) \text{ for all }g\in\Gamma\\
\sum_g\eta(g) =1
\end{array} \right\}.
\end{equation}
From a computational standpoint, the three necessary criteria are easy to check as properties of the coefficient function $\eta$; in fact, no matrix multiplication is required until we wish to check idempotence. The following proposition reduces the idempotence condition to a property of $\eta$ as well.
\begin{prop}[Idempotence  in group algebra characterized by convolution identity]\label{prop:convInvariance}
Given a finite group $\Gamma$ with right regular representation $\{\rrep{g}\}_{g\in \Gamma}$ and a binary function $\eta:\Gamma\to\Z_2$, the matrix $\sum_g \eta(g)\rrep{g}$ is idempotent if and only if $\eta$ is invariant under convolution with itself; that is, if and only if $\eta(h)=\eta\convl \eta(h):=\sum_g \eta(g)\eta(g^{-1}h)$ for each $h\in\Gamma$.
\end{prop}
\begin{proof}

Let $\Gamma$ and $\{\rrep{g}\}_{g\in \Gamma}$ be as given above and $\eta:\Gamma\to\Z_2$ be a binary function. We note that
\begin{equation}\label{eq:squareConv}
    \left(\vphantom{\sum}\right.\sum_{g\in \Gamma} \eta(g)\rrep{g}\left.\vphantom{\sum}\right)^2
        =\sum_{g_1,g_2\in \Gamma} \eta(g_1)\eta(g_2)\rrep{g_1g_2}
        =\sum_{g,h\in \Gamma} \eta(g)\eta(g^{-1}h)\rrep{h};
\end{equation}
it follows that $\sum \eta(g)\rrep{g}=\big(\sum \eta(g)\rrep{g}\big)^2$ implies $\eta(h)=\sum_g \eta(g)\eta(g^{-1}h)$ for each $h\in \Gamma$.
On the other hand, suppose $\eta:\Gamma\to\Z_2$ is convolution invariant. Then
\[
\sum_{h\in \Gamma} \eta(h)\rrep{h} 
    =\sum_{h\in \Gamma}\lrphantombrackets{[}{]}{\sum}{\sum_{g\in \Gamma} \eta(g)\eta(g^{-1}h)}\rrep{h}
    =\sum_{g,h\in \Gamma} \eta(g)\eta(g^{-1}h)\rrep{h},
\]
which by equation \eqref{eq:squareConv} is equal to $\big(\sum \eta(g)\rrep{g}\big)^2$, and the proof is complete.
\end{proof}

\begin{ex}
Consider $D_3$, the dihedral group of order 6, described $\langle a,b:a^3=1,b^2=1,b^{-1}ab=a^{-1}\rangle$; ordering the elements $1,a,a^2,b,ab,a^2b$, then the right regular representation matrices of $D_3$ are given by 
\[\underset{\rrep{1}}{\left[
\begin{smallmatrix}
1& \lite{0}& \lite{0}& \lite{0}& \lite{0}& \lite{0}\\
\lite{0}& 1& \lite{0}& \lite{0}& \lite{0}& \lite{0}\\
\lite{0}& \lite{0}& 1& \lite{0}& \lite{0}& \lite{0}\\
\lite{0}& \lite{0}& \lite{0}& 1& \lite{0}& \lite{0}\\
\lite{0}& \lite{0}& \lite{0}& \lite{0}& 1& \lite{0}\\
\lite{0}& \lite{0}& \lite{0}& \lite{0}& \lite{0}& 1  
\end{smallmatrix}
\right]},\underset{\rrep{a}}{
\left[
\begin{smallmatrix}
\lite{0}& 1& \lite{0}& \lite{0}& \lite{0}& \lite{0}\\
\lite{0}& \lite{0}& 1& \lite{0}& \lite{0}& \lite{0}\\
1& \lite{0}& \lite{0}& \lite{0}& \lite{0}& \lite{0}\\
\lite{0}& \lite{0}& \lite{0}& \lite{0}& \lite{0}& 1\\
\lite{0}& \lite{0}& \lite{0}& 1& \lite{0}& \lite{0}\\
\lite{0}& \lite{0}& \lite{0}& \lite{0}& 1& \lite{0}  
\end{smallmatrix}
\right]},\underset{\rrep{a^2}}{
\left[
\begin{smallmatrix}
\lite{0}& \lite{0}& 1& \lite{0}& \lite{0}& \lite{0}\\
1& \lite{0}& \lite{0}& \lite{0}& \lite{0}& \lite{0}\\
\lite{0}& 1& \lite{0}& \lite{0}& \lite{0}& \lite{0}\\
\lite{0}& \lite{0}& \lite{0}& \lite{0}& 1& \lite{0}\\
\lite{0}& \lite{0}& \lite{0}& \lite{0}& \lite{0}& 1\\
\lite{0}& \lite{0}& \lite{0}& 1& \lite{0}& \lite{0}  
\end{smallmatrix}
\right]},\underset{\rrep{b}}{
\left[
\begin{smallmatrix}
\lite{0}& \lite{0}& \lite{0}& 1& \lite{0}& \lite{0}\\
\lite{0}& \lite{0}& \lite{0}& \lite{0}& 1& \lite{0}\\
\lite{0}& \lite{0}& \lite{0}& \lite{0}& \lite{0}& 1\\
1& \lite{0}& \lite{0}& \lite{0}& \lite{0}& \lite{0}\\
\lite{0}& 1& \lite{0}& \lite{0}& \lite{0}& \lite{0}\\
\lite{0}& \lite{0}& 1& \lite{0}& \lite{0}& \lite{0}  
\end{smallmatrix}
\right]},\underset{\rrep{ab}}{
\left[
\begin{smallmatrix}
\lite{0}& \lite{0}& \lite{0}& \lite{0}& 1& \lite{0}\\
\lite{0}& \lite{0}& \lite{0}& \lite{0}& \lite{0}& 1\\
\lite{0}& \lite{0}& \lite{0}& 1& \lite{0}& \lite{0}\\
\lite{0}& \lite{0}& 1& \lite{0}& \lite{0}& \lite{0}\\
1& \lite{0}& \lite{0}& \lite{0}& \lite{0}& \lite{0}\\
\lite{0}& 1& \lite{0}& \lite{0}& \lite{0}& \lite{0}  
\end{smallmatrix}
\right]},\underset{\rrep{a^2b}}{
\left[
\begin{smallmatrix}
\lite{0}& \lite{0}& \lite{0}& \lite{0}& \lite{0}& 1\\
\lite{0}& \lite{0}& \lite{0}& 1& \lite{0}& \lite{0}\\
\lite{0}& \lite{0}& \lite{0}& \lite{0}& 1& \lite{0}\\
\lite{0}& 1& \lite{0}& \lite{0}& \lite{0}& \lite{0}\\
\lite{0}& \lite{0}& 1& \lite{0}& \lite{0}& \lite{0}\\
1& \lite{0}& \lite{0}& \lite{0}& \lite{0}& \lite{0}  
\end{smallmatrix}
\right]}.
\]
A quick check for convolution invariance among the twelve coefficient functions satisfying the conditions in \eqref{set:etaProps} shows that only $I_6$ and $G_1:=\rrep{1}+\rrep{a}+\rrep{a^2}$ give suitable Gramians. Synthesis matrices for the two classes are given by $\syMtx[G_1]:=\left[\begin{smallmatrix}1&1&1&0&0&0\\0&0&0&1&1&1 \end{smallmatrix}\right]$ and $\syMtx[I_6]=I_6$.  
\end{ex}

It follows from Proposition \ref{prop:convInvariance} that the coefficient function of the Gramian of a binary Parseval group frame is always convolution invariant, but convolution invariance of such a function does not ensure matrix symmetry:
\begin{ex}
Let $\{\rrep{j}\}_{j=0}^6$ be the right regular representation matrices for the group $\Z_7$, noting that $\rrep{j}^*=\rrep{j}^{-1}=\rrep{-j}$. Consider the coefficient function given by
\[
    \eta(x)=\begin{cases} 
   1 & \text{if } x\in\{0,1,2,4\}  \\
   0       & \text{if } x\in\{3,5,6\}
  \end{cases},
\]
which is easily verified to satisfy $\eta=\eta\convl\eta$. It is clear, however, that the matrix $G=\sum_{j=0}^6 \eta(j)\rrep{j}$ is not symmetric (since $\eta(1)\neq \eta(6)$, for example), so $G$ is not the Gramian of any frame.
\end{ex}

Adding idempotence under convolution to the conditions in \eqref{set:etaProps} removes the need to require that the coefficient function sums to $1$,
which is then implicit in $\eta(e)=1$.
We conclude a characterization of the coefficient functions of binary Parseval group frames.

\begin{thm}[Gramians of binary Parseval $\Gamma$-frames characterized by $\eta$]\label{thm:coefficientChar}
Given a finite group $\Gamma$ with right regular representation matrices $\{\rrep{g}\}_{g\in \Gamma}$
and  $G=\sum_g \eta(g)\rrep{g}$, then $G$ is the Gramian of a binary Parseval $\Gamma$-frame if and only if $\eta(e)=1$, $\eta$ is symmetric under inversion
of its argument and idempotent under convolution.
\end{thm} 

\begin{proof}
Since $G=\sum_g \eta(g)\rrep{g}$ and $\rrep{g}^*=\rrep{g^{-1}}$ for all $g\in\Gamma$, it follows that $G$ is symmetric if and only if $\eta$ is. Further, Proposition \ref{prop:convInvariance} equates the idempotence of $G$ with that of $\eta$.  Now, $\eta(e)=1$ if and only if $\eta(g)=1$ for all $g\in\Gamma$, if and only if $G$ has at least one odd column.

Theorem \ref{thm:GramCharacterization} provides four conditions which characterize the Gramians of binary Parseval group frames, three of which we have just demonstrated are equivalent to conditions on $\eta$. Since $G$ automatically satisfies the remaining condition as an element of the group algebra $\Z_2[\{\rrep{g}\}]$, it follows that $G=\sum_g \eta(g)\rrep{g}$ is the Gramian of a binary Parseval $\Gamma$-frame if and only if $\eta(e)=1$, $\eta=\eta*\eta$, and $\eta(g)=\eta(g^{-1})$ for all $g\in\Gamma$.
\end{proof}

In light of the last theorem, we can replace the necessary conditions (\ref{set:etaProps}) with necessary and sufficient conditions
for G being the Gramian of a binary Parseval $\Gamma$-frame $\mathcal F$,
\begin{equation}\label{set:etaPropsB}
G\in
\left\{\sum_{g\in\Gamma}\eta(g)\rrep{g}\;
\begin{array}{|l}
\eta(e)=1\\
\eta(g)=\eta(g^{-1}) \text{ for all }g\in\Gamma\\
\eta=\eta\convl\eta
\end{array} \right\},
\end{equation}
where $\eta$ is assumed to be a $\Z_2$-valued function on $\Gamma$.

\subsection{Binary Parseval frames from orbits of abelian groups}Next, we focus on the special case of abelian groups.

\begin{lemma}[Idempotence from square root condition for abelian groups]\label{lem:convInvariant}
Given a finite abelian group $\Gamma$ 
and function $\eta: \Gamma \to \Z_2$, $\eta$ is idempotent under convolution if and only if 
\[
\eta(g)=\sum_{h^2=g}\eta(h)\textSpace{for all}g\in\Gamma.
\]
\end{lemma}
\begin{proof}
 Fix $g\in\Gamma$ and partition $\Gamma$ into $K_g:=\{h\in\Gamma:h^2=g\}$ and $B:=\Gamma\backslash K_g$. Since $\Gamma$ is abelian and by the definition of $B$, we have that for each element $x\in B$ there is a unique element $x^{-1}g=gx^{-1}\in B$, and $x\neq x^{-1}g$.  We refine our partition on $\Gamma$ by separating $B$ into disjoint sets $B_1$ and $B_2$ such that no two elements  $x,y\in B_i$ multiply to $g$, arbitrarily assigning one element from each pair $\{x,x^{-1}g\}$ to $B_1$ and the other to $B_2$. 
 
 The idempotence under convolution is thus expressed

 \begin{align*}
 \eta(g)= &  \sum_{h\in \Gamma} \eta(h)\eta(h^{-1}g)\\
  = &\sum_{h\in K_g} \eta(h)\eta(\underbrace{h^{-1}g}_{=h}) + \sum_{x\in B_1} \eta(x)\eta(x^{-1}g)+\sum_{y\in B_2}\eta(y)\eta(y^{-1}g)\\
    =&\sum_{h\in K_g} \eta(h)\eta(h) + \sum_{x\in B_1}\big[ \eta(x)\eta(x^{-1}g)+\eta(x^{-1}g)\eta(\underbrace{x}_{\mathclap{=(x^{-1}g)^{-1}g}})\big]\\
    =&\sum_{h\in K_g} \eta(h)+2\sum_{x\in B_1} \eta(x)\eta(x^{-1}g)\\
    =&\sum_{h\in K_g} \eta(h),
 \end{align*}
where the last two identities follow from noting $z^2=z$ and $2z=0$ for all $z\in\Z_2$.
\end{proof}
\begin{ex}[Binary Parseval group frames of $\Z_6$]\label{ex:bpgfZ6}
We use the preceding lemma to classify the binary Parseval group frames generated by the (abelian) additive group $\Gamma:=\Z_6$. Suppose $G=\sum \eta(g)\rrep{g}$ is the Gramian of a binary Parseval $\Z_6$-frame $\Fm$; in the notation of the proof of Lemma \ref{lem:convInvariant}, we have $K_1=K_3=K_5=\emptyset$, for which the ``square root condition'' asserts $\eta(1)=\eta(3)=\eta(5)=0$. By the coefficient function characterization of the Gramian, $\eta(0)=1$, and since $2+2=4$, we have that either $\eta(2)=\eta(4)=1$ or $G$ is the identity matrix. It follows, noting that both options induce idempotent matrices, that any binary Parseval $\Z_6$-frame has a Gram matrix that is either $I_6$ or $G:=\rrep{0}+\rrep{2}+\rrep{4}$, 
\[
    \underset{G}{
    \left[\begin{smallmatrix}         
1&\lite{0}&1&\lite{0}&1&\lite{0}\\
\lite{0}&1&\lite{0}&1&\lite{0}&1\\
1&\lite{0}&1&\lite{0}&1&\lite{0}\\
\lite{0}&1&\lite{0}&1&\lite{0}&1\\
1&\lite{0}&1&\lite{0}&1&\lite{0}\\
\lite{0}&1&\lite{0}&1&\lite{0}&1  
    \end{smallmatrix}\right]
    }
    =
\underset{\rrep{0}}{
    \left[\begin{smallmatrix}
    1&\lite{0}&\lite{0}&\lite{0}&\lite{0}&\lite{0}\\
\lite{0}&1&\lite{0}&\lite{0}&\lite{0}&\lite{0}\\
\lite{0}&\lite{0}&1&\lite{0}&\lite{0}&\lite{0}\\
\lite{0}&\lite{0}&\lite{0}&1&\lite{0}&\lite{0}\\
\lite{0}&\lite{0}&\lite{0}&\lite{0}&1&\lite{0}\\
\lite{0}&\lite{0}&\lite{0}&\lite{0}&\lite{0}&1  
\end{smallmatrix}\right]
}
+
\underset{\rrep{2}}{
    \left[\begin{smallmatrix}        
\lite{0}&\lite{0}&1&\lite{0}&\lite{0}&\lite{0}\\
\lite{0}&\lite{0}&\lite{0}&1&\lite{0}&\lite{0}\\
\lite{0}&\lite{0}&\lite{0}&\lite{0}&1&\lite{0}\\
\lite{0}&\lite{0}&\lite{0}&\lite{0}&\lite{0}&1\\
1&\lite{0}&\lite{0}&\lite{0}&\lite{0}&\lite{0}\\
\lite{0}&1&\lite{0}&\lite{0}&\lite{0}&\lite{0}  
    \end{smallmatrix}\right]
     }
     +
     \underset{\rrep{4}}{
    \left[\begin{smallmatrix}
\lite{0}&\lite{0}&\lite{0}&\lite{0}&1&\lite{0}\\
\lite{0}&\lite{0}&\lite{0}&\lite{0}&\lite{0}&1\\
1&\lite{0}&\lite{0}&\lite{0}&\lite{0}&\lite{0}\\
\lite{0}&1&\lite{0}&\lite{0}&\lite{0}&\lite{0}\\
\lite{0}&\lite{0}&1&\lite{0}&\lite{0}&\lite{0}\\
\lite{0}&\lite{0}&\lite{0}&1&\lite{0}&\lite{0}  
     \end{smallmatrix}\right]
    }.
\]
To complete the classification, we note that $G$ and $I_6$ represent distinct classes, since the Gramians of switching equivalent binary Parseval frames have the same number of nonzero entries. Synthesis matrices for the two classes are given by $\syMtx[G]:=\left[\begin{smallmatrix}1&0&1&0&1&0\\0&1&0&1&0&1 \end{smallmatrix}\right]$ and $\syMtx[I_6]=I_6$.
\end{ex}

In the special case that every element in a group has exactly one square root, an even stronger consequence holds for $\eta$.  This ``unique square root'' property is determined solely by the parity of a group's order (see, for example, Proposition 2.1 in \cite{uniqueSqRts}), and we recall part of this characterization in the following lemma.

\begin{lemma}
A finite group of odd order has unique square roots.  
\end{lemma}
\begin{proof}
Let $\Gamma$ be a group such that $\abs{\Gamma}=2n-1$ for some integer $n\geq2$, and suppose $a^2=b^2$ for some $a,b\in\Gamma$.  Then $a^{2n-1}=e$, so $a=a\cdot a^{2n-1}=a^{2n}=b^{2n}=b$. 
\end{proof}

\begin{thm}[Odd-ordered abelian groups and $\eta$]\label{thm:uniqueRoots}
Let $\Gamma$ be a finite abelian group of odd order. Then the map $g\mapsto\{g'\in\Gamma:g^{2^m}=g'\text{ for some }m\in\N\}$ partitions $\Gamma$, and a function $\eta: \Gamma \to \Z_2$ is idempotent under convolution if and only if $\eta$ is constant on these sets.
\end{thm}
\begin{proof}
Since $\Gamma$ has odd order, the unique square root property reduces the condition
\[
\eta(g)=\sum_{h^2=g}\eta(h)\textSpace{for all}g\in\Gamma
\]
to
\[
\eta(g^2)=\eta(g)\textSpace{for all}g\in\Gamma;
\]
thus, it remains only to show that the map defined in the hypothesis partitions $\Gamma$.
Let $g\in\Gamma$, and for $j\in\N$, define $\gamma_j:=g^{2^j}$. Since $\Gamma$ finite, we may take $N$ to be the least positive integer such that $\gamma_{N+1}\in\{\gamma_j\}_{j=1}^{N}$. By $\Gamma$'s unique square roots, it follows that $\gamma_{N+1}=\gamma_1$, or $g=g^{2^{N}}$. 
Now, let $h\in\Gamma$ be distinct from $g$, and similarly define a sequence by $\hat{\gamma}_j:=h^{2^j}$, with minimal $M$ such that $h=h^{2^{M}}$.  It follows that either $\{\gamma_j\}_{j\in\N}=\{\hat{\gamma_j}\}_{j\in\N}$ or $\{\gamma_j\}_{j\in\N}\cap\{\hat{\gamma_j}\}_{j\in\N}=\emptyset$, and the claim is shown.
\end{proof}

\begin{ex}[Classes of $\Z_{17}$-frames]
Suppose $G=\sum_g\eta(g)\rrep{g}\in\Z_2[\Z_{17}]$ is the Gramian of a binary Parseval $\Z_{17}$-frame. $\Z_{17}$ satisfies the conditions of Theorem \ref{thm:uniqueRoots}, so we know that $\eta$ is constant on each of the sets $\Delta_1:=\{1,2,4,8,16,15,13,9\}$ and $\Delta_3:=\{3,6,12,7,14,11,5,10\}$, which are closed under inversion.  Thus, $G$ is one of exactly four operators, given by 
\[
\begin{array}{cccc}
I_{17},&
I_{17}+\sum_{j\in\Delta_1}\rrep{j},&
I_{17}+\sum_{j\in\Delta_3}\rrep{j},&
\text{ and }\sum_{j\in\Z_{17}}\rrep{j}.
\end{array}
\]
\end{ex}

In illustrating an application of Theorem \ref{thm:uniqueRoots}, this example also motivates us to introduce some additional notation.

\begin{defn}[Symmetric doubling orbit, symmetric doubling orbit partition, $\rrep{\sdo{g}}$]\label{def:sdo}
Let $\Gamma$ be a finite abelian group having unique square roots. For any element $g\in\Gamma$, the \emph{symmetric doubling orbit} of $g$ is the set
\[
    \sdo{g}:=\{h\in\Gamma: g^{2^m}=h \text{ for some }m\in\N\}\cup\{h\in\Gamma: (g^{-1})^{2^m}=h \text{ for some }m\in\N\}.
\]
We define 
\[
    \rrep{\sdo{g}}:=\sum_{h\in\sdo{g}}\rrep{h}
\]
and say that the collection $\Gamma'=\{\sdo{g}\}_{g\in J}$ is the \emph{symmetric doubling orbit partition} of $\Gamma$ (indexed by representatives $J\subset \Gamma$) if $\bigcup_{g\in J}\sdo{g}=\Gamma$ and for distinct $g,h\in J$ we have $\sdo{g}\neq\sdo{h}$.
\end{defn}
\begin{rem}
We comment on our terminology.
Since $\Gamma$ is abelian, let us momentarily consider it as an additive group and express it as $\Gamma\cong\bigoplus_{i=1}^k\Z_{p_i}$.
Modifying the notation in Definition \ref{def:sdo} accordingly, we have
\[
\sdo{g}:=
    \{h\in\Gamma: 2^mg=h \text{ for some }m\in\N\}\cup\{h\in\Gamma: 2^m(-g)=h \text{ for some }m\in\N\},
\]
which is equivalent to
$\left\{\rho_m g\right\}_{m=1}^L\cup\left\{\rho_m(-g)\right\}_{m=1}^L$ for some $L\in\N$, where $\rho_m:=2^m I_k$.

It is easy to verify that the matrices $\{2^m I_k\}_{m=1}^L$ are representation matrices for the multiplicative subgroup generated by $2$ in $\Z_L$, 
which motivates the ``doubling orbit'' part of the name \emph{symmetric doubling orbit}: $\left\{\rho_m g\right\}_{m=1}^L$ is, in fact, the orbit of $g$ under the action of $\left\langle 2\right\rangle^{\times}_{\Z_L}$.
\end{rem}
We proceed with two results making use of the new notation. The first may be considered a corollary of Theorems \ref{thm:coefficientChar} and \ref{thm:uniqueRoots}, and the second uses the symmetric doubling orbit partitioning to provide a count of the binary Parseval $\Gamma$-frame unitary equivalence classes for our specified groups $\Gamma$.

\begin{thm}[Characterization of binary Parseval $\Gamma$-frames for odd order, abelian $\Gamma$]\label{thm:charBySDOsums}
Let  $\Gamma$ be an odd-ordered abelian group  with  right regular representation $\rrep{}$ and symmetric doubling orbit partition $\{\sdo{g}\}_{g\in J}$. Let $G$ be a linear map $G:\Z_2^{\Gamma}\to\Z_2^{\Gamma}$, then $G$ is the Gramian of a binary Parseval $\Gamma$-frame if and only if $G=\sum_{g\in J}
\nu(\sdo g)\rrep{\sdo{g}}$ for some $\nu:\Gamma'\to\Z_2$ with $\nu([e])=1$.
\end{thm}
\begin{proof}
Assume $G$ is the Gramian of a binary Parseval $\Gamma$-frame, and let $G=\sum_{g\in\Gamma}\eta(g)\rrep{g}$; then $\eta$ is idempotent under convolution (by Theorem \ref{thm:coefficientChar}) and thus constant on symmetric doubling orbits (by Theorem \ref{thm:uniqueRoots}).  It follows that $\nu(\sdo{g}):= \eta(g)$
is well defined and
satisfies $G=\sum_{\sdo{g}\in J}\nu(\sdo{g})\rrep{\sdo{g}}$ and $\nu(\sdo e)=1$.

Conversely, assume $G=\sum_{g\in J}\nu(\sdo g)\rrep{\sdo{g}}$ for some $\nu$ such that $\nu(\sdo e)=1$, and define $\eta:\Gamma\to\Z_2$ by assigning 
$\eta(g) = \nu(\sdo g)$, then the conditions of Theorem~\ref{thm:uniqueRoots} are met and $\eta$ is idempotent under convolution. Noting that $\eta(e)=1$, the conditions of Theorem~\ref{thm:coefficientChar} hold as well, and $G$ is thereby the Gramian of a binary Parseval $\Gamma$-frame.
\end{proof}
\begin{cor}[Enumerating unitary equivalence classes of binary Parseval $\Gamma$-frames]\label{cor:qtyGrams}
Let the \linebreak group $\Gamma$ and the set $\Gamma'$ 
be as above and define $k:=\abs{\Gamma}$, $k':=\abs{\Gamma'}$, then the number of Gramians of unitarily inequivalent binary Parseval $\Gamma$-frames is $2^{k'-1}\leq2^{\frac{1}{2}(k-1)}$.
\end{cor}
\begin{proof}
The value $2^{\abs{\Gamma'}-1}$ is the number of functions $\nu:\Gamma'\to\Z_2$ having the property that 
$\nu(\sdo e)=1$, thus enumerating the functions delineated in Theorem \ref{thm:charBySDOsums}. The quantity $2^{\frac{1}{2}(k-1)}$ is achieved if $\Gamma=\Z_3^2$, as well as any other case such that $\abs{\sdo{g}}=2$ for all $g\in\Gamma\backslash\{e\}$. Exceeding this bound implies the existence of $h\in\Gamma\backslash\{e\}$ such that $\abs{\sdo{h}}=1$, which implies $h=h^2$.  Since the only idempotent element of a group is the identity element,
such an $h$ does not exist.
\end{proof}

Results in \cite{BodmannCampMahoney2014} justify the use of Gramians as class representatives of binary Parseval frames. For a group of size $k$, the naive upper bound of $2^{k^2}$ binary matrices thereby drops to $2^{\frac{1}{2}(k^2-1)}$ symmetric binary matrices with at least one odd column. Theorem~\ref{thm:gramInAlgebra} in this paper puts our Gramians in $\Z_2[\{R_g\}]$, a set of order $2^k$. In the case of abelian $\Gamma$ with unique square roots, Corollary \ref{cor:qtyGrams} gives the number of distinct Gramians of binary Parseval $\Gamma$-frames exactly as $2^{\abs{\Gamma'}-1}$, where $\abs{\Gamma'}\leq\frac{1}{2}(k+1)$ is the quantity of symmetric doubling orbits of $\Gamma$.
     Thus, for a given abelian group $\Gamma$ of odd order $k$, the unitary equivalence classes of binary Parseval frames are classified by computing the ranks of $2^{\abs{\Gamma'}-1}\leq2^{\frac{1}{2}(k-1)}$ Gramians.

Writing the elements of $\Z_p^q$ as vectors suggests plotting subsets of the group for visualization purposes.  Noting that inverse elements are obtained by multiplying by $-1\mod p$, the fact that each symmetric doubling orbit is a collection of scalar multiples of a single element puts each of the points of a given symmetric doubling orbit on a line in $\Z_p^q$ containing the origin.

For many odd-prime/natural-number pairs $p,q$, in fact, the nontrivial symmetric doubling orbits of $\Z_p^q$ are each identical to that line, minus the origin; this property holds any time the multiplicative subgroup of $\Z_p$ generated by 2 is $\Z_p\backslash\{0\}$, as in the cases of $p\in\{3,5,11,13\}$.  It also occurs when $\abs{\langle2\rangle_{\Z_p}^{\times}}=\frac{1}{2}(p-1)$ and $(-1)\notin\langle2\rangle_{\Z_p}^{\times}$, since the \emph{symmetric} part completes the set; the smallest $p$ for which this occurs is 7.  

The work in this paper shows that any Gramian in the group algebra of the regular representations 
yields a binary Parseval $\Z_p^q$-frame for $q\in\N$ and odd prime $p$ if the group elements represented in the sum are the union of a collection of these linear subspaces.  However, the converse of this statement is not true,
as each Mersenne prime (that is, having the form $2^n-1$) greater than 7 provides a counter example, as does every Fermat prime (i.e., of the form $2^n+1$) greater than 5. We illustrate this in Figure~\ref{fig:Z17sq} with plots of the symmetric doubling orbits of $\Z_p^2$ for the smallest value that demonstrates this behavior, $p=17$. Each plot shows a pair of orbits (one in red, one in black) that partition a line into two subsets. Any of the $2^{36}$ linear combinations of coefficients that are constant on these symmetric doubling orbits 
represents a distinct Gramian of a binary Parseval $\Z_{17}^2$-frame.

\begin{figure}[h!]
\includegraphics[trim = 15mm 10mm 10mm 10mm,clip,width=0.75\paperwidth]{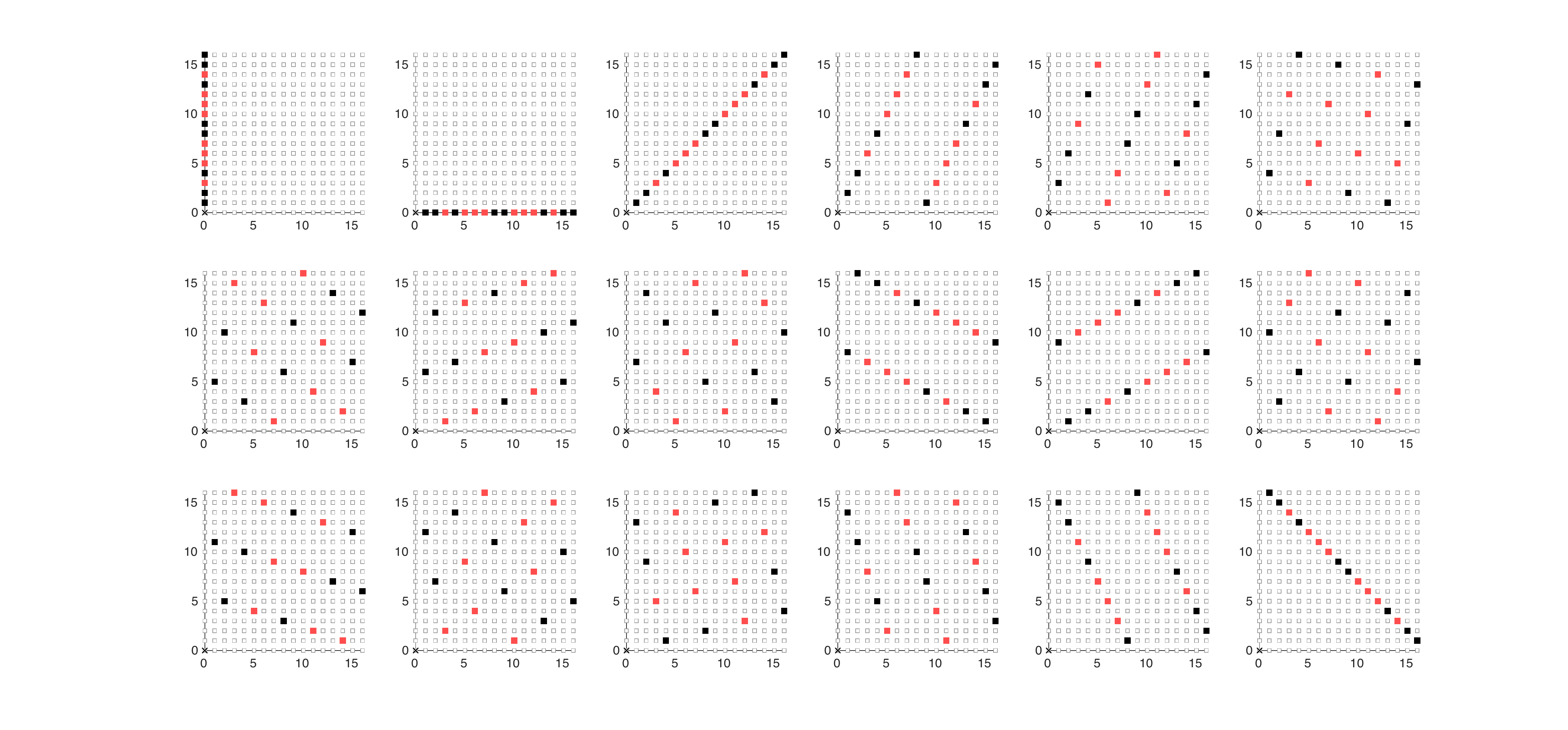}
\caption{Symmetric doubling orbits of $\Z_{17}^2$, plotted in pairs that are complements in one-dimensional subspaces of $\Z_{17}^2$.}\label{fig:Z17sq}
\end{figure}

	\subsection{An algorithm for classifying binary Parseval $\Gamma$-frames for abelian $\Gamma$ of odd order}\label{sec:methodsThy}\label{sec:application}
	
	For groups of smallest order, the unitary equivalence classes are manageable. 
	However, even for $\Z_3^3$
	the enumeration of Parseval frames becomes too tedious to do by hand. 
	One reason is that group automorphisms may lead to
	different Gramians. The resulting set could be reduced to one representative without losing structural information. 
	We recall that switching offers
	a coarser equivalence relation that is suitable for removing copies obtained by group automorphisms.
	
	\begin{prop}[Automorphisms on $\Gamma$ and automorphic switching equivalence]\label{prop:permutationEquivGrams}
		Let $\Gamma$ be a finite group with right regular repesentation matrices $\{\rrep{g}\}_{g\in\Gamma}$, and let $\Fm:=\{f_g\}_{g\in\Gamma}$ be a binary Parseval $\Gamma$-frame with Gramian $G:=\sum_g\eta(g)\rrep{g}$,
		{then an operator $H$ is the Gramian of a binary Parseval $\Gamma$-frame that is automorphically switching equivalent to $\Fm$ if and only if $H=\sum_g\eta(\sigma(g))\rrep{g}$ for some $\sigma\in\Aut{\Gamma}$.}
	\end{prop}
\begin{proof}
	{Let $\rho$ be the group representation that induces $\Fm$; we first show that the composition of the coefficent function $\eta$ with an automorphism induces the Gramian of an automorphically  switching equivalent  frame.
	
	Let $\sigma\in\Aut{\Gamma}$ and define $H:=\sum_g\eta(\sigma(g))\rrep{g}$. From
		}
		$\rho\circ\sigma$ being a group homomorphism, it follows that $\{\rho_{\sigma(g)}f_e\}_{g\in\Gamma}$ is a binary Parseval $\Gamma$-frame that is 
		automorphically switching equivalent to $\Fm$.  By Corollary \ref{cor:gammaFrameIFFGramInAlgebra}, the Gramian $G'$ of $\{\rho_{\sigma(g)}f_e\}_{g\in\Gamma}$ admits a coefficient function $\nu$ such that $G'=\sum_g\nu(g)\rrep{g}$. It remains only to prove that $\nu=\eta\circ\sigma$, so that $G'=H$.
		
	Recall from the proof of Theorem \ref{thm:gramInAlgebra} that for $a,b\in\Gamma$, $G_{a,b}=\eta(a^{-1}b)$ and $G'_{a,b}=\nu(a^{-1}b)$. We conclude
	\begin{align*}
		\nu(a^{-1}b)
		&=\bindotp{\rho_{\sigma(b)}f_e}{\rho_{\sigma(a)}f_e}\\
		&=G_{\sigma(a),\sigma(b)}\\
		&=\eta\left(\sigma(a)^{-1}\sigma(b)\right)\\
		&=\eta\left(\sigma(a^{-1}b)\right),
	\end{align*}
	this last identity following from the fact that $\sigma$ is an automorphism.
	
	{
	Conversely, suppose $\Fm':=\{f'_g\}_{g\in\Gamma}$ is a $\Gamma$-frame that is automorphically switching equivalent to a frame $\Fm$ induced by a representation $\rho$. Let the unitary $U$ and $\sigma\in\Aut{\Gamma}$ give $f_g'=Uf_{\sigma(g)}=U\rho_{\sigma(g)}f_e$ for all $g\in\Gamma$. 
	Let the Gramians of $\Fm$ and $\Fm'$ be $G= \sum_g\eta(g)\rrep{g}$ and $H=\sum_g\nu(g)\rrep{g}$, respectively. We equate
		\begin{align*}
		\nu(a^{-1}b)
		&=\bindotp{f_b'}{f_a'}\\
		&=\bindotp{U\rho_{\sigma(b)}f_e}{U\rho_{\sigma(a)}f_e}\\
		&=\bindotp{\rho_{\sigma(b)}f_e}{\rho_{\sigma(a)}f_e}\\
		&=\eta\left(\sigma(a^{-1}b)\right),
	\end{align*}
	and we see that $H=\sum_g\eta(\sigma(g))\rrep{g}$ has the claimed form.}
\end{proof}

Next, we study how symmetric doubling orbits  behave under automorphisms.  Let $g,h\in\Gamma$ and $a\in\N$ such that $g=h^{2^a}$. Under an
automorphism $\sigma\in\Aut{\Gamma}$, we identify $\sigma(g)=\sigma(h^{2^a})=\sigma(h)^{2^a}$. Consequently, if $g \in \sdo h$, then $\sigma(g) \in \sdo{\sigma(h)}$.
This means the action of $\sigma$ on $\Gamma$ passes to an action on the symmetric doubling orbits.

\begin{defn}For a finite abelian group $\Gamma$ partitioned into symmetric doubling orbits $\Gamma'=\{\sdo g \}_{\sdo{g} \in J}$ and an automorphism $\sigma$,
we let $\tilde \sigma$ be the associated bijection on $\Gamma'$ such that $\tilde \sigma(\sdo g) = \sdo{\sigma(g)}$.
\end{defn}

\begin{cor}[Automorphisms on $\Gamma$ and symmetric doubling orbits]\label{cor:automorphismsPreserveSDOs}
	Let $\Gamma$, $\{R_g\}$, $\Fm$, $G$ and $\eta$ be as above, and suppose $\Gamma$ is abelian of odd order. Let $\Gamma'$ be the symmetric doubling orbit partition of $\Gamma$, and $G=\sum_{ \sdo g\in \Gamma'}\tilde \eta(\sdo g)\rrep{\sdo{g}}$ with $\tilde \eta: \Gamma' \to \mathbb{Z}_2$, then
	{an operator $H$ is the Gramian of a binary Parseval $\Gamma$-frame that is automorphically switching equivalent to $\Fm$ if and only if 
	$H=\sum_{\sdo g\in\Gamma'}\tilde \eta(\tilde \sigma(\sdo g))R_{\sdo{g}}$ for some $\sigma\in\Aut{\Gamma}$.}
\end{cor}
\begin{proof}
	%
	{Let $\sigma\in\Aut{\Gamma}$, and
	$\tilde \sigma$ the associated bijection on $\Gamma'$. Let $\eta(g) = \tilde \eta(\sdo g)$, for any $g \in \Gamma$.
	}
	Consequently,
	$\sum_{\sdo g\in\Gamma'}\eta(\tilde \sigma(\sdo g))R_{\sdo{g}}=\sum_{g\in\Gamma}\eta(\sigma(g))R_{g}$. Applying Proposition \ref{prop:permutationEquivGrams} completes the proof.
\end{proof}

	By identifying Gramians in the group algebra with functions on the group, Corollary \ref{cor:gammaFrameIFFGramInAlgebra} reduces the search for Gram matrices associated with a given $\Gamma$ to a search over a subset of $\Z_2$-valued coefficient functions on $\Gamma$; Theorem \ref{thm:coefficientChar} specifies that subset. Proposition \ref{prop:permutationEquivGrams} allows a classification of the valid coefficient functions in terms of the automorphism group on $\Gamma$. Specialized results for abelian groups summarized in Corollary \ref{cor:automorphismsPreserveSDOs} provide us with a concrete method for obtaining all binary Parseval group frames for abelian, odd-ordered groups.
	The following result provides theoretical justification for an algorithm guaranteed to produce a list of Gram matrices
	that contains exactly one representative from each automorphic switching equivalence class. 
\begin{thm} \label{prop:multiplicationTableJustification}
Given an odd-ordered abelian group $\Gamma$ and a set $\cl{M}$ which generates the automorphism group of $\Gamma$, the algorithm described below partitions the Gramians of binary Parseval $\Gamma$-frames under automorphic switching equivalence.
\end{thm}
\begin{proof}
Let $\Gamma'$ be the symmetric doubling orbit partitioning of $\Gamma$. Corollary \ref{cor:automorphismsPreserveSDOs} reduces the theorem's partitioning to the comparison of symmetric doubling orbit coefficient functions. In particular, two binary Parseval $\Gamma$-frames are automorphically switching equivalent  if and only if their Gramians $\sum_{\sdo g\in\Gamma'}\eta(\sdo g)R_{\sdo g}$ and 
$\sum_{\sdo g\in\Gamma'}\nu(\sdo g)R_{\sdo g}$ have the property that $\eta(\sdo g)=\nu(\tilde \sigma(\sdo g))$ for all $g\in\Gamma$ and $\tilde \sigma$ determined by the action of some $\sigma\in\Aut{\Gamma}$ on the symmetric doubling orbits.  Given a coefficient function $\eta$, the algorithm does one of two things each time it accesses the multiplication table: it either identifies another coefficient function belonging to the same partition as $\eta$, or it terminates the search for coefficient functions in that partition. It thus remains to show that the algorithm exhausts the partition for any such $\eta$.

Let $\eta:\Gamma\to\Z_2$ be constant on symmetric doubling orbits. Enumerate $\cl{M}=\{M_i\}_{i=1}^N$ and define $M_0:=Id\in\aut(\Gamma)$, and let $\accSet{0}:=\{S_{\eta}\}$, where $S_{\eta}:=\eta^{-1}(1)$. 
For $j\in\N$, define the set collection
\[
    \accSet{j}:=\{M_i(S):i=1,2,\ldots,N\text{ and } S\in\accSet{j-1}\}.
\]
Note that $\accSet{j-1}\subseteq\accSet{j}$ for all $j\in\N$, since $S\in\accSet{j-1}$ implies that $M_0(S)\in\accSet{j}$. The algorithm produces each $\accSet{j}$ sequentially and terminates the search for elements in $\eta$'s partition at the end of identifying the elements of $\accSet{j}$ if $\accSet{j}=\accSet{j-1}$.  Now, if $\nu(g)=\eta(\sigma(g))$ for all $g\in\Gamma$ and some $\sigma\in\Aut{\Gamma}$, then there is a finite sequence $l_1,l_2,\ldots,l_k$ such that $\sigma=M_{l_k}M_{l_{k-1}}\cdots M_{l_1}$. It follows that the partition reprepresented by $\eta$ is the set $\accSet{L}$ for some $L\in\N$; thus, the algorithm produces the partition of $\eta$ if and only if there is an integer $j_{\eta}$ such that 
\begin{equation}\label{eq:accSetContainment}
    \accSet{1}\subsetneq\accSet{2}\subsetneq\cdots\subsetneq\accSet{j_{\eta}}=\accSet{j_{\eta}+i}\textSpaceHalf{for all }i\in\N.
\end{equation}
Let $j_0\in\N$ be such that $\accSet{j_0-1}=\accSet{j_0}$; existence follows from the finiteness of $\Z_2[\Gamma]$.
To prove that such $j_{\eta}$ exists, it is enough to show that the equality $\accSet{j_0}=\accSet{j_0-1}$ implies $\accSet{j_0}=\accSet{j_0+i}$ for all $i\in\N$.

Let $S'\in\accSet{j_0+1}$. By the inclusion $\accSet{j_0}\subseteq\accSet{j_0+1}$, it is left to show that $S'\in\accSet{j_0}$.  By the definition of $\accSet{j_0+1}$, we have $S'=M_i(S)$ for some $i\in\{0,1,\ldots,N\}$ and some $S\in\accSet{j_0}=\accSet{j_0-1}$. Since $S\in\accSet{j_0-1}$, it follows that $S'=M_i(S)\in\accSet{j_0}$, and the proof is complete.
\end{proof}

    \smallskip
  
\noindent%
\setlength{\fboxsep}{6pt}   
\begin{minipage}{.8\paperwidth} 
\itemindent0pt
\framebox[1.1\width][c]{\begin{minipage}{.69\paperwidth}
\centering{\textbf{A Practitioner's Guide to Generating Gramians of Binary Parseval $\Gamma$-Frames for abelian $\Gamma$ of Odd Order}}
\begin{enumerate}[leftmargin=.1in]
    \item Produce a set $J$ so that $\{e\}\cup J$ indexes the symmetric doubling orbit partition $\Gamma'$ of $\Gamma$. 
    \item Select $\mathcal{M}\subset\Aut{\Gamma}$ to seed a multiplication table. If $\mathcal{M}$ generates $\Aut{\Gamma}$, this algorithm provides 
    a partition of $\Gamma$-frames into automorphic switching equivalence classes. \emph{(See Remark \ref{rem:autChoice})}
    \item Produce the automorphism multiplication table containing a row for each $M_i \in \mathcal M$, with entry $(i,j)$ giving $M_i(\sdo{g_j})$.
    \item For each $m \leq\frac{1}{2}\abs{\Gamma'}$, apply the method described in Example \ref{ex:classifyZ3up2} to partition 
    subsets of 
    the collection 
    $\{\bigcup_{g\in K}\sdo{g}:K\subset J,\abs{K}=m\}$. For $m>\frac{1}{2}\abs{\Gamma'}$, use the fact that for given indexing sets $K,K'$, the sets $\bigcup_{g\in K}\sdo{g}$ and $\bigcup_{g\in K'}\sdo{g}$ represent the same class if and only if $\bigcup_{g \in J\setminus K}\sdo{g}$ and $\bigcup_{g \in J \setminus K'}\sdo{g}$ do.
\end{enumerate}
\end{minipage}}
\end{minipage}

\begin{rem}[Sampling $\Aut{\Gamma}$]\label{rem:autChoice}
Choosing $\mathcal{M}=\Aut{\Gamma}$ guarantees accurate partitioning, although $\Aut{\Gamma}$ may be difficult to calculate.  
Theorem~\ref{prop:multiplicationTableJustification} tells us that we can obtain this partitioning as long as $\mathcal{M}$ is a generating set for $\Aut{\Gamma}$.  If $\mathcal{M}$ is not known to generate $\Aut{\Gamma}$, the potential undersampling of the automorphism group may simply lead to the case that some classes are represented multiple times; the number of Gramians is still smaller than $2^{\abs{\Gamma'}-1}$.
\end{rem}

		The following example demonstrates how the algorithm works.
	
	\begin{ex}[Classifying binary Parseval $\Z_3^2$-frames]\label{ex:classifyZ3up2}

Let the symmetric doubling orbit partition of $\Z_3^2$ given by set $\Gamma'=\{\sdo g: g \in \{ e \} \cup J\}$ with
 $J:=\{\nCksm{1}{0},\nCksm{1}{1},\nCksm{0}{1},\nCksm{1}{2}\}$. We shall classify binary Parseval $\Z_3^2$-frames up to automorphic switching equivalence by identifying suitable Gramian representatives for each class. These Gramians have the form $I+\sum_{i=1}^m\rrep{\sdo{g_i}}$ for some $m\in\{0,1,2,3,4\}$ and distinct $g_i$'s, and we proceed by considering one value of $m$ at a time. We make use of the fact that for any finite vector space $V$, $\Aut{V}\equiv\GL{V}$. 

$\mathbf{m=0:}$ The cases of $m=0$ and $m=4$ are trivial and listed in the summary.

$\mathbf{m=1:}$ The matrix $\left[\begin{smallmatrix}1&1\\1&0\end{smallmatrix}\right]\in\GL{\Z_3^2}$ gives
\[
\left[\begin{smallmatrix}1&1\\1&0\end{smallmatrix}\right]\sdo{\nCksm{1}{0}}
	=\left\{
	    \left[\begin{smallmatrix}1&1\\1&0\end{smallmatrix}\right]\nCksm{1}{0},
	    \left[\begin{smallmatrix}1&1\\1&0\end{smallmatrix}\right]\nCksm{2}{0}
        \right\}
    =\left\{\nCksm{1}{1},\nCksm{2}{2} \right\}
    =\sdo{\nCksm{1}{1}};
\]
applying the preceding corollary, $I+\rrep{\sdo{\nCksm{1}{0}}}$ and $I+\rrep{\sdo{\nCksm{1}{1}}}$ are thus Gramians of automorphically switching equivalent binary Parseval $\Z_3^2$-frames.  With this in mind, consider the multiplication table given in Table~\ref{fig:Z3up2_MT}. 
\begin{table}[h]
    \centering
\[\begin{array}{r|cccc}
	    &	\sdo{\nCksm{1}{0}}&	\sdo{\nCksm{1}{1}}&	\sdo{\nCksm{1}{2}}&	\sdo{\nCksm{0}{1}}\\
	\hline			
\left[\begin{smallmatrix}1&1\\1&0\end{smallmatrix}\right]&	\sdo{\nCksm{1}{1}}&	\sdo{\nCksm{1}{2}}&	\sdo{\nCksm{0}{1}}&	\sdo{\nCksm{1}{0}}\\
\left[\begin{smallmatrix}2&1\\1&0\end{smallmatrix}\right]&	\sdo{\nCksm{1}{2}}&	\sdo{\nCksm{0}{1}}&	\sdo{\nCksm{1}{1}}&	\sdo{\nCksm{1}{0}}\\
\left[\begin{smallmatrix}1&1\\0&1\end{smallmatrix}\right]&	\sdo{\nCksm{1}{0}}&	\sdo{\nCksm{1}{2}}&	\sdo{\nCksm{0}{1}}&	\sdo{\nCksm{1}{1}}
\end{array}
\]    
    \caption{Multiplication table for selected $M\in\GL{\Z_3^2}$}
    \label{fig:Z3up2_MT}
\end{table}
The first row shows that for $g,h\in J$,  $\sdo{g}=\left[\begin{smallmatrix}1&1\\1&0\end{smallmatrix}\right]^a\sdo{h}$ for some integer $a$. It follows that the four operators $I+\rrep{\sdo{g}}$ represent the same automorphic switching equivalence class.

$\mathbf{m=2:}$ Similarly, the first two entries in the first row give
\[\left[\begin{smallmatrix}1&1\\1&0\end{smallmatrix}\right]
    \left(\sdo{\nCksm{1}{0}}\cup\sdo{\nCksm{1}{1}}\right)
    =\left[\begin{smallmatrix}1&1\\1&0\end{smallmatrix}\right]\sdo{\nCksm{1}{0}}
    \cup\left[\begin{smallmatrix}1&1\\1&0\end{smallmatrix}\right]\sdo{\nCksm{1}{1}}
    =\sdo{\nCksm{1}{1}}\cup\sdo{\nCksm{1}{2}},
\]
implying
\[
I+\rrep{\sdo{\nCksm{1}{0}}}\!\!+\rrep{\sdo{\nCksm{1}{1}}}\;\text{ and }\;\;
I+\rrep{\sdo{\nCksm{1}{1}}}\!\!+\rrep{\sdo{\nCksm{1}{2}}}
\]
are representatives of the same equivalence class.
Proceeding down the first two columns, we find that Gramians
$I+\rrep{\sdo{\nCksm{1}{2}}}+\rrep{\sdo{\nCksm{0}{1}}}$ and 
$I+\rrep{\sdo{\nCksm{1}{0}}}+\rrep{\sdo{\nCksm{1}{2}}}$ represent that same class.

Reentering the table with the index pair $\nCksm{1}{2},\nCksm{0}{1}$, we find the sets $\sdo{\nCksm{0}{1}}\cup\sdo{\nCksm{1}{0}}$ and $\sdo{\nCksm{0}{1}}\cup\sdo{\nCksm{1}{1}}$; it follows that each of the six distinct Gramians $I+\sum_{i=1}^2\rrep{\sdo{g_i}}$ represent the same class. \emph{Note: If this step had not exhausted the ``$m=2$'' case, we would continue to reenter the multiplication table with each new equivalent $\bigcup g_i$ until the class stops growing.}

$\mathbf{m=3:}$ We make use of set complements. Fixing $g,h\in J$, let $a$ satisfy $\sdo{g}=\left[\begin{smallmatrix}1&1\\1&0\end{smallmatrix}\right]^a\sdo{h}$. It follows that $\bigcup_{g'\neq g}\sdo{g'}=\left[\begin{smallmatrix}1&1\\1&0\end{smallmatrix}\right]^a\bigcup_{h'\neq h}\sdo{h'}$, since each $M\in\GL{\Z_3^2}$ is a bijection on $\bigcup_{g'\in J}\sdo{g'}$. We conclude that each of the sums $I+\sum_{i=1}^3\rrep{\sdo{g_i}}$ represents the same equivalence class, since $g$ and $h$ were chosen arbitrarily.

\textbf{Summary:} The binary Parseval $\Z_3^2$-frames partition into five automorphic switching equivalence classes, with representative Gramians given by the identity operator, the $9\times9$ matrix of $1$'s, and three more representatives
		\[
		I+\rrep{\sdo{\nCksm{1}{0}}},\;\; 
		I+\rrep{\sdo{\nCksm{1}{0}}}\!\!+\rrep{\sdo{\nCksm{0}{1}}},\text{ and }\;\;
		I+\rrep{\sdo{\nCksm{1}{0}}}\!\!+\rrep{\sdo{\nCksm{1}{1}}}\!\!+\rrep{\sdo{\nCksm{0}{1}}}.
		\]


    
    Hence, the nontrivial Gramians turn out to have ranks $3$ ($m=1$), $5$ ($m=2$), and $7$ ($m=3$). 
    The Gram matrices belonging to a given rank are equivalent, so
    we partitioned the fourteen nontrivial Gramians $I+\sum_{i=1}^m\rrep{\sdo{g_i}}$ into three equivalence classes.
    
    \end{ex}



\section{Binary Parseval group frames as codes}\label{sec:framesAsCodes}
One motivating application of binary Parseval group frames is their use as codes. The range of the analysis operator $\anMtx[]$ is the so-called code book
in $\mathbb{Z}_2^k$. Each codeword $y$ in this codebook is the image of a unique vector $x \in \Z_2^n$ which is obtained by $x=\anMtx[]^* y$.

When $k>n$, the redundancy introduced by the embedding $\anMtx[]$ makes it possible to accurately recover $x$ from a \emph{corrupted} codeword $\tilde{y}:=E\anMtx[]x+\epsilon$, provided the diagonal error matrix $E$ and error vector $\epsilon$ are known to meet certain specified conditions.
	
For our binary case, we consider two types of errors: \emph{erasures} ($\tilde{y}=E\anMtx[]x$, $E_{i,i}\in\{0,1\}$) and \emph{bit-flips} ($\tilde{y}=\anMtx[]x+\epsilon$, $\epsilon\in\Z_2^J$). We say that a binary Parseval frame $\Fm$ is \emph{robust to $m$ erasures} if for every diagonal binary matrix $E$ having at most $m$ zeros on the diagonal, the operator $E\anMtx[]$ admits a left inverse.  This is equivalent to the condition that the Hamming distance between any two vectors in the image of $\anMtx[]$ (or, equivalently, of the Gramian of $\Fm$) is at least $m+1$, since any pair of vectors that differ in only $m$ entries are indistinguishable if those entries are ``erased.''  By the linearity of  $\anMtx[]$, this is also equivalent to the condition that each nonzero vector in $\anMtx[]\Z_2^n$ has weight exceeding $m$.

On the other hand, we say that $\Fm$ is \emph{robust to $m$ bit-flips} if $\normz{\anMtx[]x_1-\anMtx[]x_2}\geq2m+1$ for all $x_1,x_2\in\Z_2^n$, $x_1 \ne x_2$. This notion of ``robustness to error'' implies the ability to identify each vector in the set $B:=\{\anMtx[]x+\epsilon:x\in\Z_2^n,\normz{\epsilon}\leq m\}$ as the (corrupted) image of a unique vector in $\Z_2^n$. Note that if $\normz{\anMtx[]x_1-\anMtx[]x_2}=2m$ for some pair $x_1,x_2\in\Z_2^n$, then there exist $m$-weighted error vectors $\epsilon_1$ and $\epsilon_2$ such that $\anMtx[]x_1+\epsilon_1=\anMtx[]x_2+\epsilon_2$.  Now suppose that for any distinct $y_1,y_2\in\anMtx[]\Z_2^n$, we have $\normz{y_1-y_2}\geq 2m+1$, and let $\tilde{y}\in B$; by the triangle inequality, there is exactly one point $y\in\anMtx[]\Z_2^n$ such that $\normz{y-\tilde{y}}\leq m$.  Thus, we may recover the intended signal $y$ by identifying the nearest point in $\anMtx[]\Z_2^n$ to $\tilde{y}$, and recovery of $x=\syMtx[]y$ follows.
	
	Again appealing to the linearity of $\Theta$,
both robustness conditions are expressed in terms of the minimum weight among nonzero vectors in the range of $\Theta$.
For Parseval frames,  the range of the analysis operator coincides with that of the Gramian, so it can be stated equivalently in terms of
the range of the Gramian.

\begin{defn}[Code weight of a Gramian or frame]
		Given an operator $G:\Z_2^J\rightarrow\Z_2^J$, the \emph{code weight} of $G$ is the value $\min_{y\in G(\Z_2^J)\setminus\{0\}}\normz{y}$. 
		
\end{defn}

In the following section, we compare $\Z_p^q$-frames with $\Z_{p^q}$-frames. The final major result in this paper is a proof that every binary Parseval $\Z_{p^q}$-frame is switching equivalent to a $\Z_p^q$-frame. We also include a number of examples in which the classes of $Z_{p^q}$-frames are mapped to their switching equivalent $\Z_p^q$ frames for select $p$'s and $q$'s and show that in addition to subsuming binary Parseval $\Z_{p^q}$-frames, there are examples of $\Z_p^q$-frames that outperform them as codes. 

\subsection{Comparing frames generated with $\Z_{p^q}$ vs.\ $\Z_p^q$}\label{sec:Examples} 
Fix $m\in\N$, and let $\Fm_1$ and $\Fm_2$ be switching equivalent binary Parseval frames. Theorem 4.9 in \cite{BodmannCampMahoney2014} establishes that this equivalence implies that $\Fm_1$ is robust to $m$ erasures if and only if $\Fm_2$ is. By the Gramian code weight characterization of robustness to each type of error, it follows that $\Fm_1$ is robust to $m$ bit-flips if and only if $\Fm_2$ is. Theorem 4.11 in \cite{BodmannLeRezaTobinTomforde2009} characterizes switching equivalence between binary Parseval frames as permutation equivalence between their Gramians $G_1$ and $G_2$:
	\[
	\Fm_1\sweq\Fm_2\textSpace{if and only if} G_1=P^*G_2P\text{ for some permutation matrix }P.
	\]
	Hence, for the purposes of evaluating binary Parseval group frames as codes, whether we are concerned about erasures or bit-flips, we may restrict our attention to permutation equivalence classes of the Gramians of such frames.

Applying the techniques in this paper, we have classified binary Parseval group frames for each of the groups below, using Gramians as class representatives. Recalling that the quantity of vectors in a group frame is given by the size of the group, and that the rank of the Gramian is the dimension of the inducing frame, we can directly compare the performance of several frames as error-correcting codes.  To facilitate comparing $\Z_{p^q}$ and $\Z_p^q$ for a given pair $p,q$, we combine details for the two groups in a single table; in each of the comparisons below, $\Z_p^q$-frames perform at least as well as $\Z_{p^q}$-frames, and they often outperform their $\Z_{p^q}$ counterparts. In fact, 
the exhaustive search of best performing codes associated with binary Parseval frames generated with $\Z_{p}^q$ is guaranteed to
be at least as good as the best codes generated with $\Z_{p^q}$, as shown in Theorem~\ref{thm:cyclicInProduct} below.

We now provide a sequence of results which culmintate in the proof of Theorem \ref{thm:cyclicInProduct}, which states that, given an odd prime $p$ and $q\in\N$, any binary Parseval $\Z_{p^q}$-frame is switching equivalent a binary Parseval $\Z_p^q$-frame.  
Practically, this reduces to showing that the Gramian of a binary Parseval $\Z_{p^q}$-frame satisfies the Gram characterization for a $\Z_p^q$-frame for some reindexing. We accomplish this by showing that the symmetric doubling orbits of $\Z_{p^q}$ partition those of $\Z_p^q$, in the sense that for each $n\in\Z_{p^q}$, the matrix $\rrep{\sdo{n}}$ can be written as the sum of matrices in $\{\rrep{\sdo{g}}\}_{g\in\Z_p^q}$.

The map which produces this reindexing is the inverse of the function $\indBij:\Z_p^q\to\Z_{p^q}$ given by 
\begin{equation}\label{eq:indBij}
\indBij(g):=\sum_{i=1}^qp^{i-1}g_i,
\end{equation}
where the arithmetic is carried out in $\Z_{p^q}$; this mapping is akin to converting from numbers written in base $p$. It is worth noting that for a given $i\in\{1, 2, \dots, {q-1}\}$ and $g\in\Z_p^q$, we have that $p^i$ divides $\indBij(g)$ if and only if the first $i$ entries of $g$ are  zero; if $p^i$ divides $\indBij(g)$ and $p^{i+1}$ does not, then the $j$-th entry of $g$, denoted $g_j$, is nonzero.

We recall a few fundamental properties of finite multiplicative groups in the context of this work.  For a given $n\in\N$, we may consider $\Z_n$ as the ring $(\Z_n,\cdot,+)$, in which case the subset of elements having multiplicative inverses forms the multiplicative group $\mgZpq[n]:=(\Z/n\Z)^{\times}$. The elements of $\Z_n$  that provide elements in $\mgZpq[n]$ are those coprime with $n$.

Here we shall denote the multiplicative subgroup of $\mgZpq[n]$ generated by element $k$ as $\msg{k}{n}:=\msg{k}{\Z_n}$.

\begin{prop}
Let $p,q,k\in\N$ with $p$ prime and $1< k<p$. 
Then $\abs{\msgk}=p^{q-1}\abs{\msgk[p]}$ and $x\in\msgk$ if and only if
$x\modp{p}\in\msgk[p]$.
\end{prop}
\begin{proof}
Note that $\mgZpq[p]\cong\Z_{p-1}$, since each nonzero element of $\Z_p$ is coprime with $p$. Recalling that a finite cyclic group of order $mn$ is isomorphic to $\Z_m\times\Z_n$ if $m$ and $n$ are coprime, we have that 
\[
\mgZpq\cong \Z_{p^{q-1}(p-1)} \cong \Z_{p^{q-1}}\times\Z_{p-1} \cong \Z_{p^{q-1}}\times\mgZpq[p].
\]
It follows that $\abs{\mgZpq[p]}=p-1$ and $\abs{\mgZpq}=p^{q-1}(p-1)$. We consider now the subgroups $\msgk[p]\leq\mgZpq[p]$ and $\msgk\leq\mgZpq$.

Note that $x\in\msgk$ implies that $x\modp{p}\in\msgk[p]$, so that each element in $\msgk$ may be written in the form $mp+t$ for some $m\in\{0,1,\ldots,p^{q-1}-1\}$ and some $t\in\msgk[p]$ considered as an element of $\Z$. It follows that $\abs{\msgk}\leq p^{q-1}\abs{\msgk[p]}$. We shall show equality holds by demonstrating that the reverse inequality holds; the resulting equality will imply that \emph{every} element of $\mgZpq$ of the form $mp+t$ as above is an element of $\msgk$, completing the characterization $\msgk[p]=\{x\modp{p}\mid x\in\msgk[p]\}$.

Let
\begin{align*}
\gqrest:\msgk &\to \Z_{p^{q-1}}\times\msgk[p]\\
        k^j&\mapsto{j\times k^j},
\end{align*}
where we consider $\Z_{p^{q-1}}\times\msgk[p]$ as a group in the natural way, inheriting its group operation componentwise. For $i,j\in\N$, we have
\[
\gqrest(k^i k^j)=\gqrest(k^{i+j})=(i+j,k^{i+j})=(i,k^i)(j,k^j)=\gqrest(k^i)\gqrest(k^j),
\]
and thus $\gqrest$ is a group homomorphism.
Since the orders of the cyclic groups  $\Z_{p^{q-1}}$ and $\msgk[p]$ are coprime, it follows that $\gqrest$ exhausts its range and that $\abs{\gqrest(\msgk)}=\abs{\Z_{p^{q-1}}}\abs{\msgk[p]}$. This implies $\abs{\msgk}\geq p^{q-1}\abs{\msgk[p]}$. We conclude that $\abs{\msgk}= p^{q-1}\abs{\msgk[p]}$ and that $x\in\msgk$ if and only if
$x\modp{p}\in\msgk[p]$.
\end{proof}

\begin{cor}\label{cor:cosetChar}
Let $p,q,k\in\N$ with $p$ prime and $1< k<p$.  Then for each $y\in\mgZpq$,
\[
\abs{y\msgk}=p^{q-1}\abs{\msgk[p]}
\]
and $x\in y\msgk$ if and only if
$x\modp{p}\in y\msgk[p]\modp{p}$.
\end{cor}
\begin{proof}
Fix $y\in\mgZpq$. Since the cosets of a subgroup partition a group into equal sized sets, the preceding proposition yields
$\abs{y\msgk}=\abs{\msgk}=p^{q-1}\abs{\msgk[p]}$.

Now, if $x\in y\msgk$, then $x\modp{p}\in y\msgk[p]$. Since $\abs{\{x\in\mgZpq: x\modp{p}\in y\msgk[p]\}}=p^{q-1}\abs{\msgk[p]}$,

 It follows that 
\[
y\msgk=\{x\in\mgZpq: x\modp{p}\in y\msgk[p]\},
\]
and the proof is complete.
\end{proof}

\begin{thm}\label{thm:doublingOrbitElements}
Given $p,q,k\in\N$ with $p$ prime and $1< k<p$, let $x,y\in\Z_{p^q}$ and define nonnegative integers $x',y',j_x,j_y$ such that $x=x'p^{j_x}$, $y=y'p^{j_y}$, and $p$ divides neither $x'$ nor $y'$. Then $x\in y\msgk$ if and only if $j_x=j_y$ and $x'\modp{p}\in y'\msgk[p]\modp{p}$.
\end{thm}
\begin{proof}
Suppose $x\in y\msgk$, so that $x'p^{j_x}=k^ly'p^{j_y}$ for some $l\in\N$. Then $j_x=j_y$, since $k$ and $p$ are coprime and $p$ is coprime with each of $x'$ and $y'$. Next, set $r:=j_x=j_y$ and write 
\begin{equation}\label{eq:xyDecomposed1}
x'p^r\equiv k^ly'p^r\modp{p^q}.
\end{equation}
Since $x$ and $y$ may be seen as elements in $p^r\Z_{p^p}\cong\Z_{p^{q-r}}$, we may also identify $x'$ and $y'$ as elements of $\mgZpq[p^{q-r}]\leq Z_{p^{q-r}}$ and note that congruence (\ref{eq:xyDecomposed1}) implies
\begin{equation*}
x'\equiv k^ly'\modp{p^{q-r}}.
\end{equation*}
Then, using $y'\in\mgZpq[p^{q-r}]$, Corollary \ref{cor:cosetChar} yields $x'\modp{p}\in y'\msgk[p]\modp{p}$.

Conversely, assume $j_x=j_y=:r$ and $x'\modp{p}\in y'\msgk[p]\modp{p}$. Then the conditions of Corollary \ref{cor:cosetChar} are met for $x,y\in\mgZpq[p^{q-r}]$ and $x'\equiv k^{l'}y'\;\;\modp{p^{q-r}}$ for some $l'\in\N$. Embedding $y'\msgk[p^{q-r}]$ into $Z_{p^q}$ by $g\mapsto p^r g$ for $g\in y'\msgk[p^{q-r}]$, we have that $x=x'p^r\equiv 2^{l'}y'p^r\;\modp{p^q}=2^{l'}y$.
\end{proof}

The following lemma makes precise the claim that the map $\indBij$ given by (\ref{eq:indBij}) maps the doubling orbits of $\Z_p^q$ into those of $\Z_{p^q}$.
\begin{lemma}\label{lem:phiLinear}
Given $p,q\in\N$ with $p$ an odd prime, let $x\in\Z_{p^q}$. If $g\in\indBijInv(x\msgtwo)$, then for $h\in\Z_p^q$, we have that $\indBij(g\msgtwo[p]+h)\subseteq x\msgtwo+\indBij(h)$. As a consequence, $\indBij(\indBijInv(x\msgtwo)+h)= x\msgtwo+\indBij(h)$.
\end{lemma}

\begin{proof}
We begin by proving the lemma for the case that $h=0$.
Let $x$, $p$ and $q$ be as in the hypothesis, and let $g=\indBijInv(x)$. Since $\msgtwo[p]$ is cyclic, it suffices to show that for each $g'\in\indBijInv(x\msgtwo)$ there exists $k\in\N$ such that $\indBij(2g')=2^kx$; since $\msgtwo$ is cyclic, it suffices to demonstrate this for the case $g'=g$.

If $x=0$ then $g=(0)_{i=1}^q$, and the claim is shown; assume, then, that $x\neq0$ and define nonnegative integers $x'$ and $r$ such that $x=x'p^r$ and $p$ does not divide $x'$. Note that $r$ gives the quantity of leading zeros in the sequence $(g_i)_{i=1}^q$.

 Expressing this equality in terms of the definition of $\indBij$,
\[
\sum_{i=1}^q p^{i-1}(2g_i\,\modp{p})\equiv2^kx\; \modp{p^q},
\]
where $2g_i\,\modp{p}$ is considered as an element of $\Z$. For each $i\in\numto{q}$, we may express $2g_i\,\modp{p}$ as $2g_i-\delta_i p$ for some $\delta_i\in\{0,1\}$, since the value is either $2g_i$ or $2g_i-p$. Since $i\leq r$ implies $g_i=0$, it follows that $\delta_i=0$ for such  $i$. Thus
\begin{align*}
\indBij(2g)
    &\equiv \sum_{i=1}^q p^{i-1}(2g_i-\delta_i p) \;\;\;\modp{p^q}\\
    &\equiv 2\sum_{i=1}^q p^{i-1}g_i-\sum_{i=1}^q\delta_i p^i \;\;\;\modp{p^q}\\
    &\equiv 2\indBij(g)-\sum_{i=r+1}^q\delta_i p^i \;\;\;\modp{p^q}.
\end{align*}
It follows that $2\indBij(g)-\sum_{i=r+1}^q\delta_i p^i\equiv 2\indBij(g)\;\modp{p^r}$. The conditions given in Theorem \ref{thm:doublingOrbitElements} are thus satisfied for $x$ and $\indBij{(2g)}$, implying $\indBij(2g)\in2\indBij(g)\msgtwo=2x\msgtwo=x\msgtwo$. We conclude that $\indBij(2^jg)\in x\msgtwo$ for all $j\in\N$.

We now consider the general case, letting $h\in\Z_p^q$. Again, since $\msgtwo[p]$ and $\msgtwo$ are cyclic, we may assume that $g=\indBijInv(x)$.  We must show that $\indBij(g+h)=2^j x+\indBij(h)$ for some $j$.  Similar to the $h=0$ case, we define $\delta_i'$ so that $g_i+h_i-\delta_i' p\in\{0,1,\ldots,p-1\}$ for each $i\in\numto{q}$. Recalling that the value $r$ gives the number of leading zeros of $g$, we note that $\delta_i=\delta_i'=0$ for $i\leq r$.  Then
\begin{align*}
    \indBij(g+h)
        &\equiv \sum_{i=1}^q p^{i-1}(g_i+h_i-\delta_i' p)\;\;\;\modp{p^q}\\
        &\equiv \indBij(g)+\indBij(h)-\sum_{i=1}^q \delta_i' p^i\;\;\;\modp{p^q}\\
        &\equiv x-\sum_{i=r+1}^q \delta_i' p^i+\indBij(h)\;\;\;\modp{p^q}\\
        &\in x\msgtwo+\indBij(h),
\end{align*}
by Theorem \ref{thm:doublingOrbitElements}, since $x-\sum_{i=r+1}^q \delta_i' p^i\equiv x\; \modp{p^r}$. We conclude that $\indBij(g\msgtwo[p]+h)\subseteq x\msgtwo+\indBij(h)$ for all $g\in\indBijInv(x\msgtwo)$ and $h\in\Z_p^q$.

It follows that $\indBij(\indBijInv(x\msgtwo)+h)\subseteq x\msgtwo+\indBij(h)$. Set equality follows from the fact that $\indBij$ is a bijection, since both sides of the inclusion have the same number of elements.
\end{proof}

We are ready to prove the section's main result. We wish to show that for each Gramian of a binary Parseval $\Z_{p^q}$-frame, the corresponding Gramian over $\Z_2^{\Z_p^q}$ obtained from the reindexing given by $\indBijInv$ is in the group algebra $\Z_2[\Z_p^q]$; this is sufficient to show that the underlying frame is a binary Parseval $\Z_p^q$-frame, since the Gramian retains idempotence, symmetry and the weights of range vectors under switching.

\begin{thm}\label{thm:cyclicInProduct}
Let $p,q\in\N$ with $p$ an odd prime and define $\indBij:\Z_p^q\to\Z_{p^q}$ by $\indBij(g):=\sum_{i=1}^qp^{i-1}g_i$, carrying out the arithmetic in $\Z_{p^q}$.
If $\Fm=\{f_x\}_{x\in\Z_{p^q}}$ is a binary Parseval $\Z_{p^q}$-frame for $\Z_2^n$, then $\Fm':=\{f_{\indBijInv(x)}\}_{x\in\Z_{p^q}}$ is a binary Parseval $\Z_p^q$-frame. 
\end{thm}

\begin{proof}
Let $\Fm=\{f_x\}_{x\in\Z_{p^q}}$ be a binary Parseval $\Z_{p^q}$-frame for $\Z_2^n$. Denote the frame's analysis matrix and Gramian by $\anMtx$ and $G$, respectively, and let $\anMtx[\Fm']$ and $G'$ denote those of $\Fm'$.  Since $G$ and $G'$ are switching equivalent, $G'$ inherits symmetry, idempotence and column weights from $G$. By the characterization of binary Parseval group frames given by Theorem \ref{thm:GramCharacterization},  it is then left to show that $G'$ is a element of the group algebra of the right regular representation of $\Z_p^q$, denoted $\Z_2[\{R_g'\}]$.

Let $R:=\{R_x\}_{x\in\Z_{p^q}}$ be the right regular representation of $\Z_{p^q}$ and let $\eta$ be the binary coefficient function such that $G=\sum_{x\in\Z_{p^q}}\eta(x)R_x$, as guaranteed by Theorem \ref{thm:coefficientChar}. Since $\Z_{p^q}$ is an odd-ordered abelian group and $\eta$ is idempotent under convolution, Theorem \ref{thm:uniqueRoots} provides that $\eta$ is constant on cosets of the the multiplicative subgroup $\msgtwo$. We shall demonstrate that $\indBij$ induces an isomorphism between the sets $\{\sum_{y\in x\msgtwo}R_y\}_{x\in\Z_{p^q}}$ and $\{\sum_{y\in x\msgtwo}R_{\indBijInv(y)}'\}_{x\in\Z_{p^q}}$, 

Let $\funBij:\Z_2^{Z_p^q}\to\Z_2^{Z_{p^q}}$ be defined on standard basis elements by $\funBij e_g'=e_{\indBij(g)}$, where we use the $'$ (prime) to distinguish basis elements of the domain from those in the range. We wish to show that for each $x,z\in\Z_{p^q}$, the following holds:
\begin{equation}\label{eq:isoRreg1}
    \sum_{y\in x\msgtwo}R_ye_z
    =\funBij\bigg(\sum_{y\in x\msgtwo}R_{\indBijInv(y)}'e_{\indBijInv(z)}'\bigg).
\end{equation}
As described in Section~\ref{sec:regReps}, we may explicitly express the image a function $\varphi$ under $\rrep{y}$ by $\rrep{y}\varphi: z\mapsto \varphi(y+z)$, and equation (\ref{eq:isoRreg1}) becomes
\begin{align*}\label{eq:isoRreg2}
    \sum_{y\in x\msgtwo}e_{y+z}
    &=\funBij\bigg(\sum_{y\in x\msgtwo}e_{\indBijInv(y)+\indBijInv(z)}'\bigg)\\
    &=\sum_{y\in x\msgtwo}e_{\indBij(\indBijInv(y)+\indBijInv(z))}'.
\end{align*}
Thus, we are left to show that $x\msgtwo+z=\indBij(\indBijInv(x\msgtwo)+\indBijInv(z))$ for any $x,z\in\Z_{p^q}$. Taking $h:=\indBijInv(z)$, this is exactly the content of Lemma \ref{lem:phiLinear}, and the proof is complete.
\end{proof}

We illustrate this statement with some examples.
It is worth noting ahead of the examples that there is an important distinction between the symmetric doubling orbit partitionings of $\Z_p^q$ and $\Z_{p^q}$. For an odd prime $p$, it is a simple exercise to show that each $\sdo{x}$ in $\Z_p^q$ that is not $\sdo{e}$  has the same order as the symmetric doubling orbit of $1$ in $\Z_p$.
In contrast, according to Theorem~\ref{thm:doublingOrbitElements}, $\Z_{p^q}$ partitions into $kq$ nontrivial orbits
 for some $k\in\N$, $k$ of each of $q$ different sizes. Since automorphisms preserve symmetric doubling orbits (Proposition \ref{prop:permutationEquivGrams}), they also preserve orbit size; it follows that the computational savings offered by applying Corollary \ref{cor:automorphismsPreserveSDOs} as in Example \ref{ex:classifyZ3up2} do not apply or are significantly reduced when the group under consideration is $\Z_{p^q}$. In fact, for the values of $p$ and $q$ we explore here, the $2^q$ binary Parseval $\Z_{p^q}$-frame unitary equivalence classes promised by Corollary \ref{cor:qtyGrams} coincide with automorphic switching equivalence classes.  As we note in our closing remarks regarding $\Z_{17^q}$, this does not hold in general.
Of course, the number of symmetric doubling orbits of $\Z_{p^q}$ grows linearly in $q$ and may be considered as subsets of symmetric doubling orbits of $Z_p^q$ (Theorem \ref{thm:cyclicInProduct}), whose number grows exponentially as $(p^q-1)/\abs{\sdo{g}}$ for any $g\in\Z_p^q\backslash\{e\}$.

The next step is to compute the code weight of each of the Gramians, and we pause here to emphasize the computational savings made available by the methods developed thus far. Results in \cite{BodmannCampMahoney2014} justify the use of Gramians as class representatives of binary Parseval frames as codes; for a group of size $k$, the naive upper bound of $2^{k^2}$ binary matrices thereby drops to $2^{\frac{1}{2}(k^2-1)}$ symmetric binary matrices with at least one odd column. Theorem \ref{thm:gramInAlgebra} in this paper puts our Gramians in $\Z_2[\{R_g\}]$, a set of order $2^k$. In the case of abelian $\Gamma$ with unique square roots, Corollary \ref{cor:qtyGrams} gives the number of Gramians of binary Parseval $\Gamma$-frames exactly as $2^{\abs{\Gamma'}-1}$, where $\abs{\Gamma'}\leq\frac{1}{2}(k+1)$ is the quantity of symmetric doubling orbits of $\Gamma$. 
Thus, for a given abelian group $\Gamma$ of odd order $k$, we must process no more than $2^{\abs{\Gamma'}-1}\leq2^{\frac{1}{2}(k-1)}$ Gramians to determine code weights, and these matrices can be computed directly. We may process even fewer Gramians if we reduce the set to representatives of automorphic switching equivalence classes. Note that in determining the code weight of a $k\times k$ Gramian $G$, the $2^{\rank(G)}$ vectors in the operator's range may be obtained by taking all linear combinations of up to $\rank(G)$ columns of $G$, for a total $\sum_{i=1}^{\rank(G)}\nCksm{k}{i}$ operations; comparing this combinatorial problem with the algorithm above, it is evident that computational savings result from any reduction in the quantity of Gramians we are to process.

    Let us first consider the work for ${\mathbb Z}_3^2$ and $\Z_3^3$ to illustrate this.
    
    \begin{ex} \label{ex:costbenefit}
    The nontrivial Gramians in Example \ref{ex:classifyZ3up2}  turn out to have ranks $3$ ($m=1$), $5$ ($m=2$), and $7$ ($m=3$), and thus require $\nCksm{9}{3}$, $\nCksm{9}{5}$, and $\nCksm{9}{7}$ computations to exhaust linear combinations of columns as described. For the cost of producing a $3\times4$ multiplication table and a computing a handful of table look-ups and comparisons, we partitioned the fourteen nontrivial Gramians $I+\sum_{i=1}^m\rrep{\sdo{g_i}}$ into three classes; the return on that cost in the form of having fewer Gramians to weight-check was the reduction from $4\cdot\nCksm{9}{3}+6\cdot\nCksm{9}{5}+4\cdot\nCksm{9}{7}=1236$ computations to $\nCksm{9}{3}+\nCksm{9}{5}+\nCksm{9}{7}=246$. 
    
    The group ${\mathbb Z}_3^3$ has 14 symmetric doubling orbits, including $\sdo{e}$. 
    The characterization of automorphic switching equivalence classes given by Corollary \ref{cor:automorphismsPreserveSDOs}, provides that the $2^{13}$ unique Gramians of binary Parseval $\Z_3^3$-frames reduce to only thirty representatives. The resulting computational savings are substantial even before taking into account the cost of finding code weights, which has grown to $\sum_{i=1}^{\rank(G)}\nCksm{27}{i}$ operations per Gramian.
    \end{ex}

\subsubsection{Format of comparison tables}
Each comparison table contains representatives of switching equivalence classes of binary Parseval $\Gamma$-frames for each of the groups we compare. For each such class, we provide the rank and code weight of the representing Gramian. After the first table, we exclude the trivial Gramians given by the identity and the matrix of all ones; the Gramians themselves are encoded as indexing elements of their symmetric doubling orbit summands. Whenever a class in $\Z_{p^q}$ matches the performance of a class in $\Z_p^q$, the two classes are described in the same row of the associated table.  In many such cases, the two classes represent switching equivalent frames.

In the comparison of $\Z_3^2$ and $\Z_{9}$, for example, the Gramians given by $G_1=\sum_{ g\in J_1}\rrep{\sdo{g}}$ with $J_1=\{\nCksm{0}{0},\nCksm{1}{0},\nCksm{1}{1},\nCksm{0}{1}\}\subset\Z_3^2$ and $G_2=\sum_{g\in J_2}\rrep{\sdo{g}}$ with $J_2=\{0,1\}\subset\Z_9$
 each have rank 7 and code weight 2, and thus are listed in the same row.

\begin{ex}[$\Z_{9}$ vs. $\Z_3^2$]\label{ex:Z3up2vsZ9}
The symmetric doubling orbit partitioning of $\Z_{9}$ consists of $\sdo{0}$, $\sdo{3}=\{3,6\}$, and  $\sdo{1}=\{1,2,4,5,7,8\}$. In this case, each of the four binary Parseval $\Z_9$-frames is switching equivalent to one of the five binary Parseval $\Z_3^2$-frames delineated in Example \ref{ex:classifyZ3up2}. Apart from the trivial cases of the Gramian being the identity matrix or the matrix of all $1$'s, this correspondence
\[
		I+\rrep{\sdo{\nCksm{1}{0}}}=I+\rrep{\sdo{3}}\text{ and }\;\;
		I+\rrep{\sdo{\nCksm{1}{0}}}\!\!+\rrep{\sdo{\nCksm{1}{1}}}\!\!+\rrep{\sdo{\nCksm{0}{1}}}=I+\rrep{\sdo{1}},
		\]
assumes an appropriate identification of group elements. Table~\ref{fig:Z3up2vsZ9} provides the implications for the performance of codes.
\end{ex}
\begin{table}[h]
\begin{tabular}{|c|c|c|c|}
\hline
$\begin{smallmatrix}\text{Gram}\\\text{rank}\end{smallmatrix}$& $\begin{smallmatrix}\text{Code}\\\text{weight}\end{smallmatrix}$ 
&$J\subset\Z_3^2$&$J\subset\Z_9$\\
\hline
1&1&$\{\nCksm{0}{0},\nCksm{1}{0},\nCksm{1}{1},\nCksm{0}{1},\nCksm{1}{2}\}$&$\{0,1,3\}$\\
\hline
3&3&$\{\nCksm{0}{0},\nCksm{1}{0}\}$&$\{0,3\}$\\
\hline
5&3&$\{\nCksm{0}{0},\nCksm{1}{0},\nCksm{1}{1}\}$&--\\
\hline
7&2&$\{\nCksm{0}{0},\nCksm{1}{0},\nCksm{1}{1},\nCksm{0}{1}\}$&$\{0,1\}$\\
\hline
9&1&$\{\nCksm{0}{0}\}$&$\{0\}$\\
\hline
\end{tabular}
\caption{Comparing Parseval frames with Gramians $G=\sum_{g \in J} R_{\sdo g}$ 
obtained from groups $\Z_3^2$ and $\Z_{9}$, together with their code weights. See Example \ref{ex:Z3up2vsZ9} for details.}\label{fig:Z3up2vsZ9}
\end{table}

\begin{table}[h!]
\begin{tabular}{|c|c|c|c|}
\hline
$\begin{smallmatrix}\text{Gram}\\\text{rank}\end{smallmatrix}$& $\begin{smallmatrix}\text{Code}\\\text{weight}\end{smallmatrix}$ 
&$J\subset\Z_3^3$&$J\subset\Z_{27}$\\
\hline
3 & 9& $\left\{\!  \nCkGen{0\\0\\0}\!,\! \nCkGen{0\\0\\1}\!,\! \nCkGen{0\\1\\0}\!,\! \nCkGen{0\\1\\1}\!,\! \nCkGen{0\\1\\2}\!\right\}$ & \{0,3,9\} \\ \hline
5 & 9& $\left\{\!  \nCkGen{0\\0\\0}\!,\! \nCkGen{0\\0\\1}\!,\! \nCkGen{0\\1\\0}\!,\! \nCkGen{1\\0\\0}\!,\! \nCkGen{1\\0\\1}\!,\! \nCkGen{1\\0\\2}\!,\! \nCkGen{1\\1\\0}\!,\! \nCkGen{1\\2\\0}\!\right\}$ & -- \\ \hline
7 & 6& $\left\{\!  \nCkGen{0\\0\\0}\!,\! \nCkGen{0\\0\\1}\!,\! \nCkGen{0\\1\\1}\!,\! \nCkGen{0\\1\\2}\!,\! \nCkGen{1\\0\\0}\!,\! \nCkGen{1\\0\\1}\!,\! \nCkGen{1\\0\\2}\!,\! \nCkGen{1\\1\\0}\!,\! \nCkGen{1\\1\\1}\!,\! \nCkGen{1\\2\\0}\!,\! \nCkGen{1\\2\\1}\!\right\}$ & \{0,1,9\} \\ \hline
7 & 9& $\left\{\!  \nCkGen{0\\0\\0}\!,\! \nCkGen{0\\0\\1}\!,\! \nCkGen{0\\1\\0}\!,\! \nCkGen{0\\1\\2}\!,\! \nCkGen{1\\0\\1}\!,\! \nCkGen{1\\1\\0}\!,\! \nCkGen{1\\1\\1}\!\right\}$ & -- \\ \hline
9& 3& $\left\{\!  \nCkGen{0\\0\\0}\!,\! \nCkGen{0\\0\\1}\!\right\}$ & \{0,9\} \\ \hline 
9 & 6& $\left\{\!  \nCkGen{0\\0\\0}\!,\! \nCkGen{0\\1\\0}\!,\! \nCkGen{0\\1\\2}\!,\! \nCkGen{1\\0\\2}\!,\! \nCkGen{1\\1\\0}\!,\! \nCkGen{1\\1\\1}\!\right\}$ &  -- \\ \hline
9 & 8& $\left\{\!  \nCkGen{0\\0\\0}\!,\! \nCkGen{0\\0\\1}\!,\! \nCkGen{0\\1\\0}\!,\! \nCkGen{1\\0\\0}\!,\! \nCkGen{1\\0\\1}\!,\! \nCkGen{1\\0\\2}\!,\! \nCkGen{1\\1\\0}\!,\! \nCkGen{1\\1\\1}\!,\! \nCkGen{1\\1\\2}\!,\! \nCkGen{1\\2\\0}\!\right\}$ & -- \\ \hline
11 & 3& $\left\{\!  \nCkGen{0\\0\\0}\!,\! \nCkGen{0\\0\\1}\!,\! \nCkGen{0\\1\\1}\!,\! \nCkGen{0\\1\\2}\!,\! \nCkGen{1\\0\\0}\!,\! \nCkGen{1\\0\\1}\!,\! \nCkGen{1\\1\\0}\!,\! \nCkGen{1\\1\\1}\!,\! \nCkGen{1\\2\\1}\!\right\}$ & -- \\ \hline
11 & 6& $\left\{\!  \nCkGen{0\\0\\0}\!,\! \nCkGen{0\\0\\1}\!,\! \nCkGen{0\\1\\0}\!,\! \nCkGen{0\\1\\2}\!,\! \nCkGen{1\\0\\1}\!\right\}$ & -- \\ \hline
11 & 6& $\left\{\!  \nCkGen{0\\0\\0}\!,\! \nCkGen{0\\0\\1}\!,\! \nCkGen{0\\1\\0}\!,\! \nCkGen{0\\1\\2}\!,\! \nCkGen{1\\0\\0}\!,\! \nCkGen{1\\0\\2}\!,\! \nCkGen{1\\1\\0}\!,\! \nCkGen{1\\1\\2}\!,\! \nCkGen{1\\2\\0}\!\right\}$ & -- \\ \hline
13 & 3& $\left\{\!  \nCkGen{0\\0\\0}\!,\! \nCkGen{0\\0\\1}\!,\! \nCkGen{0\\1\\0}\!,\! \nCkGen{1\\0\\1}\!,\! \nCkGen{1\\0\\2}\!,\! \nCkGen{1\\1\\0}\!,\! \nCkGen{1\\1\\2}\!,\! \nCkGen{1\\2\\0}\!\right\}$ & -- \\ \hline
13 & 4& $\left\{\!  \nCkGen{0\\0\\0}\!,\! \nCkGen{0\\0\\1}\!,\! \nCkGen{0\\1\\0}\!,\! \nCkGen{0\\1\\1}\!,\! \nCkGen{0\\1\\2}\!,\! \nCkGen{1\\0\\0}\!,\! \nCkGen{1\\0\\1}\!,\! \nCkGen{1\\0\\2}\!,\! \nCkGen{1\\1\\0}\!,\! \nCkGen{1\\1\\2}\!,\! \nCkGen{1\\2\\0}\!,\! \nCkGen{1\\2\\1}\!\right\}$ & -- \\ \hline
13 & 6& $\left\{\!  \nCkGen{0\\0\\0}\!,\! \nCkGen{0\\1\\0}\!,\! \nCkGen{0\\1\\2}\!,\! \nCkGen{1\\0\\0}\!\right\}$ & -- \\ \hline
13 & 6& $\left\{\!  \nCkGen{0\\0\\0}\!,\! \nCkGen{0\\0\\1}\!,\! \nCkGen{0\\1\\0}\!,\! \nCkGen{0\\1\\2}\!,\! \nCkGen{1\\0\\0}\!,\! \nCkGen{1\\1\\0}\!,\! \nCkGen{1\\1\\2}\!,\! \nCkGen{1\\2\\0}\!\right\}$ & -- \\ \hline
15 & 3& $\left\{\!  \nCkGen{0\\0\\0}\!,\! \nCkGen{0\\0\\1}\!,\! \nCkGen{0\\1\\0}\!\right\}$ & -- \\ \hline
15 & 3& $\left\{\!  \nCkGen{0\\0\\0}\!,\! \nCkGen{0\\0\\1}\!,\! \nCkGen{0\\1\\0}\!,\! \nCkGen{1\\0\\0}\!,\! \nCkGen{1\\0\\1}\!,\! \nCkGen{1\\1\\1}\!,\! \nCkGen{1\\2\\0}\!\right\}$ & -- \\ \hline
15 & 4& $\left\{\!  \nCkGen{0\\0\\0}\!,\! \nCkGen{0\\0\\1}\!,\! \nCkGen{0\\1\\0}\!,\! \nCkGen{1\\0\\0}\!,\! \nCkGen{1\\0\\2}\!,\! \nCkGen{1\\1\\0}\!,\! \nCkGen{1\\2\\0}\!\right\}$ & -- \\ \hline
15 & 5& $\left\{\!  \nCkGen{0\\0\\0}\!,\! \nCkGen{0\\0\\1}\!,\! \nCkGen{0\\1\\0}\!,\! \nCkGen{0\\1\\1}\!,\! \nCkGen{1\\0\\0}\!,\! \nCkGen{1\\0\\1}\!,\! \nCkGen{1\\0\\2}\!,\! \nCkGen{1\\1\\0}\!,\! \nCkGen{1\\1\\1}\!,\! \nCkGen{1\\2\\0}\!,\! \nCkGen{1\\2\\1}\!\right\}$ & -- \\ \hline
17 & 3& $\left\{\!  \nCkGen{0\\0\\0}\!,\! \nCkGen{0\\0\\1}\!,\! \nCkGen{0\\1\\0}\!,\! \nCkGen{1\\0\\0}\!,\! \nCkGen{1\\0\\1}\!,\! \nCkGen{1\\1\\1}\!\right\}$ & -- \\ \hline
17 & 3& $\left\{\!  \nCkGen{0\\0\\0}\!,\! \nCkGen{0\\0\\1}\!,\! \nCkGen{0\\1\\1}\!,\! \nCkGen{0\\1\\2}\!,\! \nCkGen{1\\0\\0}\!,\! \nCkGen{1\\0\\1}\!,\! \nCkGen{1\\0\\2}\!,\! \nCkGen{1\\1\\0}\!,\! \nCkGen{1\\1\\1}\!,\! \nCkGen{1\\2\\1}\!\right\}$ & -- \\ \hline
17 & 4& $\left\{\!  \nCkGen{0\\0\\0}\!,\! \nCkGen{0\\0\\1}\!,\! \nCkGen{0\\1\\0}\!,\! \nCkGen{0\\1\\1}\!,\! \nCkGen{0\\1\\2}\!,\! \nCkGen{1\\0\\1}\!\right\}$ & -- \\ \hline
19 & 2& $\left\{\!  \nCkGen{0\\0\\0}\!,\! \nCkGen{0\\0\\1}\!,\! \nCkGen{0\\1\\0}\!,\! \nCkGen{0\\1\\1}\!,\! \nCkGen{0\\1\\2}\!,\! \nCkGen{1\\0\\0}\!,\! \nCkGen{1\\0\\1}\!,\! \nCkGen{1\\0\\2}\!,\! \nCkGen{1\\1\\0}\!,\! \nCkGen{1\\1\\1}\!,\! \nCkGen{1\\1\\2}\!,\! \nCkGen{1\\2\\0}\!,\! \nCkGen{1\\2\\1}\!\right\}$ & \{0,1,3\} \\ \hline
19 & 3& $\left\{\!  \nCkGen{0\\0\\0}\!,\! \nCkGen{0\\0\\1}\!,\! \nCkGen{0\\1\\0}\!,\! \nCkGen{1\\0\\0}\!,\! \nCkGen{1\\1\\1}\!\right\}$ & -- \\ \hline
19 & 3& $\left\{\!  \nCkGen{0\\0\\0}\!,\! \nCkGen{0\\0\\1}\!,\! \nCkGen{0\\1\\0}\!,\! \nCkGen{1\\0\\0}\!,\! \nCkGen{1\\0\\1}\!,\! \nCkGen{1\\0\\2}\!,\! \nCkGen{1\\1\\0}\!,\! \nCkGen{1\\1\\2}\!,\! \nCkGen{1\\2\\0}\!\right\}$ & -- \\ \hline
21 & 2& $\left\{\!  \nCkGen{0\\0\\0}\!,\! \nCkGen{0\\0\\1}\!,\! \nCkGen{0\\1\\0}\!,\! \nCkGen{0\\1\\2}\!\right\}$ & \{0,3\} \\ \hline
21 & 3& $\left\{\!  \nCkGen{0\\0\\0}\!,\! \nCkGen{0\\0\\1}\!,\! \nCkGen{0\\1\\0}\!,\! \nCkGen{0\\1\\1}\!,\! \nCkGen{1\\0\\0}\!,\! \nCkGen{1\\0\\1}\!,\! \nCkGen{1\\1\\1}\!,\! \nCkGen{1\\2\\0}\!\right\}$ & -- \\ \hline
23 & 2& $\left\{\!  \nCkGen{0\\0\\0}\!,\! \nCkGen{0\\1\\0}\!,\! \nCkGen{0\\1\\2}\!,\! \nCkGen{1\\0\\0}\!,\! \nCkGen{1\\0\\2}\!,\! \nCkGen{1\\1\\0}\!,\! \nCkGen{1\\1\\1}\!\right\}$ & -- \\ \hline
25 & 2& $\left\{\!  \nCkGen{0\\0\\0}\!,\! \nCkGen{0\\0\\1}\!,\! \nCkGen{0\\1\\1}\!,\! \nCkGen{0\\1\\2}\!,\! \nCkGen{1\\0\\0}\!,\! \nCkGen{1\\0\\1}\!,\! \nCkGen{1\\1\\0}\!,\! \nCkGen{1\\1\\1}\!,\! \nCkGen{1\\2\\0}\!,\! \nCkGen{1\\2\\1}\!\right\}$ & \{0,1\} \\ \hline
\end{tabular}
\caption{Comparing Parseval frames with Gramians $G=\sum_{g \in J} R_{\sdo g}$ obtained from groups $\Z_3^3$ and $\Z_{27}$, together with their code weights. See Example \ref{ex:Z3up3vsZ27}.}
\label{fig:Z3up3vsZ27}
\end{table}

\begin{ex}[$\Z_{27}$ vs. $\Z_3^3$, see Table~\ref{fig:Z3up3vsZ27}]\label{ex:Z3up3vsZ27}
The symmetric doubling orbit partitioning of $\Z_{27}$ consists of $\sdo{0}$, $\sdo{9}=\{9,18\}$, $\sdo{3}=\{3,6,12,15,21,24\}$, and $\sdo{1}=\Z_{27}\backslash(\sdo{0}\cup\sdo{9}\cup\sdo{3})$. The eight resulting Gramians  each represent a distinct automorphic switching equivalence class of binary Parseval $\Z_{27}$-frames.  The group $\Z_3^3$, as mentioned in Example~\ref{ex:costbenefit}, has 13 nontrivial symmetric doubling orbits and generates 30 automorphic switching equivalence classes of binary Parseval group frames.  
\end{ex}

\begin{ex}[$\Z_{125}$ vs. $\Z_5^3$, see Table~\ref{fig:Z5up3vsZ125}]\label{ex:Z5up3vsZ125}
The symmetric doubling orbit partitioning of $\Z_{125}$ consists of $\sdo{0}$ and three orbits, having orders  $\abs{\sdo{25}}=4$, $\abs{\sdo{5}}=20$, $\abs{\sdo{1}}=100$. As with with the other $\Z_{p^q}$ cases thus far, the symmetric doubling orbits of $\Z_{125}$ are invariant under automorphism on $\Z_{125}$. 
It follows that the eight distinct Gramians induced by the three nontrivial symmetric doubling orbits represent eight distinct classes of binary Parseval $\Z_{125}$-frames. 

The symmetric doubling orbit partitioning of $\Z_5^3$ consists of $\sdo{e}$ and 31 orbits of order 4. The $2^{31}$ distinct Gramians, each representing a distinct unitary equivalence class of binary Parseval $\Z_5^3$-frames, reduce to 7152 automorphic switching equivalence classes. Obtained by applying the algorithm described in this paper implemented in Matlab \cite{MatlabRepository}, Table~\ref{table:7152} gives a breakdown of the these classes by 
the size of $J$:
\begin{table}
{\small{\begin{center}
\begin{tabular}{|r|c|c|c|c|c|c|c|c|c|c|c|c|c|c|c|c|}\hline
$|J|$ & 1&2&3&4&5&6&7&8&9&10&11&12&13&14&15 & 16\\\hline
$N_{|J|}$ &1& 1& 1& 2& 3& 5& 12& 22& 42& 92& 174& 296& 476& 669& 832& 948\\ \hline
\end{tabular}
\smallskip

\begin{tabular}{|r|c|c|c|c|c|c|c|c|c|c|c|c|c|c|c|c|}\hline
$|J|$ &17&18&19&20&21&22&23&24&25&26&27&28&29&30&31& 32\\\hline
$N_{|J|}$ &948&832&669&476&296&174&92&42&22&12&5&3&2&1&1&1\\\hline
\end{tabular}.
\caption{Number of nonzero terms $|J|$ summed in $\sum_{g \in J} R_{\sdo g}$
and number $N_{|J|}$ of resulting automorphic switching equivalence classes.}\label{table:7152}
\end{center}
}}
\end{table}
The {\tiny{\begin{tabular}{|c|c|}\hline
0\\\hline1 \\\hline
\end{tabular}}} entry corresponds to the identity matrix; the bottom row, which gives the total number of automorphically switching equivalent classes per quantity of nontrivial symmetric doubling orbit summands, sums to 7152.

For obvious reasons, we do not list representatives from each of the 7152 automorphism equivalence classes. Instead, the comparisons in Table~\ref{fig:Z5up3vsZ125} place each of the six nontrivial Gramians of binary Parseval $\Z_{125}$-frames next to a $\Z_5^3$ representative of the same rank and having maximal code weight among binary Parseval $\Z_5^3$-frames of the same dimension.
\end{ex}
\begin{table}[h]\centering
\begin{tabular}{||c||c|c||c|c|}
\hline
$\!\!\begin{smallmatrix}\text{Gram}\\\text{rank}\end{smallmatrix}\!\!$&
$\!\!\begin{smallmatrix}\text{Code}\\\text{weight}\end{smallmatrix}\!\!$ &$J\subset\Z_5^3$&
 $\!\!\begin{smallmatrix}\text{Code}\\\text{weight}\end{smallmatrix}\!\!$ 
&$J\subset\Z_{125}$\\ \hline
5& 25& $\left\{\!\nCkGen{0\\0\\0}\!,\!  \nCkGen{0\\0\\1}\!,\! \nCkGen{1\\1\\0}\!,\! \nCkGen{1\\1\\1}\!,\! \nCkGen{1\\1\\2}\!,\! \nCkGen{1\\1\\3}\!,\! \nCkGen{1\\1\\4}\!\right\}$&25&\{0,5,25\}\\ \hline 
21&25& $\left\{\!\nCkGen{0\\0\\0}\!,\!  \nCkGen{0\\0\\1}\!,\! \nCkGen{1\\0\\0}\!,\! \nCkGen{1\\0\\1}\!,\! \nCkGen{1\\0\\2}\!,\! \nCkGen{1\\0\\3}\!,\! \nCkGen{1\\1\\3}\!,\! \nCkGen{1\\1\\4}\!,\! \nCkGen{1\\2\\0}\!,\! \nCkGen{1\\2\\1}\!,\! \nCkGen{1\\3\\0}\!\right\}$&10&\{0,1,25\} \\ \hline 
25 & 25& $\left\{\!\nCkGen{0\\0\\0}\!,\!  \nCkGen{0\\1\\0}\!,\! \nCkGen{0\\1\\3}\!,\! \nCkGen{1\\0\\1}\!,\! \nCkGen{1\\0\\2}\!,\! \nCkGen{1\\1\\3}\!,\! \nCkGen{1\\1\\4}\!,\! \nCkGen{1\\2\\0}\!,\! \nCkGen{1\\2\\2}\!,\! \nCkGen{1\\3\\2}\!\right\}$ &5&\{0,25\}\\ \hline 
101 &5& $\left\{\!\nCkGen{0\\0\\0}\!,\!  \nCkGen{0\\0\\1}\!,\! \nCkGen{0\\1\\0}\!,\! \nCkGen{1\\0\\4}\!,\! \nCkGen{1\\1\\3}\!,\! \nCkGen{1\\2\\1}\!,\! \nCkGen{1\\3\\0}\!\right\}$ &2&\{0,1,5\}\\ \hline 
105 & 5& $\begin{array}{cc}
\left\{\!\nCkGen{0\\0\\0}\!,\!  \nCkGen{0\\1\\0}\!,\! \nCkGen{0\\1\\1}\!,\! \nCkGen{0\\1\\2}\!,\! \nCkGen{0\\1\\3}\!,\! \nCkGen{0\\1\\4}\!,\! \nCkGen{1\\0\\4}\!,\! \nCkGen{1\\1\\0}\!,\! \nCkGen{1\\1\\1}\!,\! \nCkGen{1\\1\\2}\!,\! \nCkGen{1\\2\\2}\!,\vphantFour\hspace{1.3em}\right.\\
\left.\hspace{1.3em}\vphantFour\nCkGen{1\\2\\3}\!,\! \nCkGen{1\\2\\4}\!,\! \nCkGen{1\\3\\1}\!,\! \nCkGen{1\\3\\2}\!,\! \nCkGen{1\\3\\3}\!,\! \nCkGen{1\\3\\4}\!,\! \nCkGen{1\\4\\0}\!,\! \nCkGen{1\\4\\1}\!,\! \nCkGen{1\\4\\2}\!,\! \nCkGen{1\\4\\3}\!,\! \nCkGen{1\\4\\4}\!\right\}\end{array}$ &2&\{0,5\}\\ \hline 
121& 2& $\!\!\!\!\begin{array}{cc}\left\{\!\nCkGen{0\\0\\0}\!,\!  \nCkGen{0\\1\\0}\!,\! \nCkGen{0\\1\\1}\!,\! \nCkGen{0\\1\\2}\!,\! \nCkGen{0\\1\\3}\!,\! \nCkGen{0\\1\\4}\!,\! \nCkGen{1\\0\\0}\!,\! \nCkGen{1\\0\\1}\!,\! \nCkGen{1\\0\\2}\!,\! \nCkGen{1\\0\\3}\!,\! \nCkGen{1\\0\\4}\!,\! \nCkGen{1\\2\\0}\!,\! \nCkGen{1\\2\\1}\!,\vphantFour\hspace{.7em}\right.\\
\left.\hspace{.7em}\vphantFour \nCkGen{1\\2\\2}\!,\! \nCkGen{1\\2\\3}\!,\! \nCkGen{1\\2\\4}\!,\! \nCkGen{1\\3\\0}\!,\! \nCkGen{1\\3\\1}\!,\! \nCkGen{1\\3\\2}\!,\! \nCkGen{1\\3\\3}\!,\! \nCkGen{1\\3\\4}\!,\! \nCkGen{1\\4\\0}\!,\! \nCkGen{1\\4\\1}\!,\! \nCkGen{1\\4\\2}\!,\! \nCkGen{1\\4\\3}\!,\! \nCkGen{1\\4\\4}\!\right\}\end{array}\!\!\!\!$ & 2&\{0,1\}\\ \hline 
\end{tabular}
\caption{Comparing groups $\Z_5^3$ and $\Z_{125}$ as generators of binary Parseval frames, best performers for each given rank of the Gramian. See Example \ref{ex:Z5up3vsZ125}.}
\label{fig:Z5up3vsZ125}
\end{table}




\end{document}